\documentclass{amsart}
\usepackage{amssymb}
\usepackage[square]{natbib}
\usepackage{epsfig}
\newtheorem{thm}{Theorem}[section]

\newtheorem*{thm*}{Theorem}   
\newtheorem{cor}[thm]{Corollary} 
\newtheorem{lem}[thm]{Lemma}\newtheorem{conj}[thm]{Conjecture}

\newtheorem{question}[thm]{Question}
\theoremstyle{definition}

\newtheorem*{rems*}{Remarks} \newtheorem{prob}{Problem}
\theoremstyle{remark}
\newtheorem{rem}[thm]{Remark}

\newcommand{\CP}{\mathbb{C\mkern1mu P}}               
\newcommand{\HP}{\mathbb{H\mkern1mu P}}               
\newcommand{\CaP}{\mathrm{Ca}\mathbb{\mkern1mu P}^2}  
\newcommand{\C}{\mathbb{C}}\newcommand{\fH}{\mathbb{H}}\newcommand{\fK}{\mathbb{K}}
\newcommand{\fF}{\mathbb{F}}
\newcommand{\RP}{\mathbb{R\mkern1mu P}}     
\newcommand{\N}{\mathbb{N}} 
\newcommand{\fN}{\mathbb{N}} \newcommand{\bbJ}{\Lambda} 
\newcommand{\Q}{\mathbb{Q}}\newcommand{\bbV}{\Upsilon} 
\newcommand{\R}{\mathbb{R}}
\newcommand{\Z}{\mathbb{Z}}
\newcommand{\tq}{\tilde{q}}
\newcommand{\tz}{\tilde{z}}\newcommand{\tp}{\tilde{p}}
\newcommand{\Sph}{\mathbb{S}}
\newcommand{\Symp}{\mathsf{Sp}}

\newcommand{\gF}{\mathsf{F}}
\newcommand{\G}{\mathsf{G}}

 \newcommand{\gS}{\mathsf{S}}
\newcommand{\gU}{\mathsf{U}}\newcommand{\gH}{\mathsf{H}}

\newcommand{\dc}{\dot{c}}

\DeclareMathOperator{\SU}{SU}

\DeclareMathOperator{\ad}{ad} \DeclareMathOperator{\cont}{cont}
\DeclareMathOperator{\diag}{diag}

\DeclareMathOperator{\Or}{\mathsf{O}}\DeclareMathOperator{\Iso}{Iso}
\DeclareMathOperator{\SO}{SO}\DeclareMathOperator{\symrank}{symrank}
\DeclareMathOperator{\corank}{corank}\DeclareMathOperator{\Par}{Par}
\DeclareMathOperator{\Rc}{R}\DeclareMathOperator{\symdeg}{symdeg}
\DeclareMathOperator{\cohom}{cohom}
\newcommand{\gT}{\mathsf{T}}

\newcommand{\gK}{\mathsf{K}}

\newcommand{\Sp}{\mathsf{Sp}}

\DeclareMathOperator{\Fix}{Fix}

\DeclareMathOperator{\Ric}{Ric}
\DeclareMathOperator{\Gr}{Gr}

\DeclareMathOperator{\diam}{diam} \DeclareMathOperator{\scal}{scal} 
\DeclareMathOperator{\rank}{rank}
\DeclareMathOperator{\spann}{span}

\newcommand{\eps}{\varepsilon}

\newcommand{\tg}{\tilde{g}}

\newcommand{\folL}{\mathcal{L}}
\newcommand{\folF}{\mathcal{F}}

\newcommand{\so}{\mathfrak{so}}

\begin{document}

%
\begin{titlepage}\title{Nonnegatively and Positively  curved manifolds}
\author{Burkhard Wilking}
\end{titlepage}
\maketitle
\setcounter{page}{1}
\setcounter{tocdepth}{0}


The aim of this paper is to survey some results on nonnegatively and positively curved
Riemannian manifolds. One of the important features of lower curvature bounds in general
is the invariance under taking Gromov Hausdorff limits. 
Celebrated structure and finiteness results  
provide a partial understanding of  the phenomena that occur while taking 
limits. These results however are not the subject of this survey 
since they are treated in other surveys of this volume. 

In this survey we take the more classical approach 
and focus on ''effective'' results. 
There are relatively few general ''effective'' structure results in the subject. 
 By Gromov's Betti number theorem 
the total Betti number of a nonnegatively curved manifold is bounded above 
by an explicit
 constant which only depends on the dimension. 
The Gromoll Meyer theorem says that a positively curved 
open manifold is diffeomorphic to the Euclidean space.
In the case of nonnegatively curved open manifolds, the soul theorem 
of Cheeger and Gromoll and Perelman's solution of the soul conjecture 
clearly belong to the greatest structure results in the subject, as well. 

Also relatively good is the understanding of fundamental groups of 
nonnegatively curved manifolds. A theorem of Synge asserts that an even
dimensional orientable compact manifold of positive sectional curvature is
simply connected. An odd dimensional positively curved manifold is known to
be orientable (Synge), and its fundamental group is finite by the classical
theorem of Bonnet and Myers. The fundamental groups of  nonnegatively curved
manifolds are virtually abelian, as a consequence of Toponogov's splitting
theorem. However, one of the ''effective'' conjectures in this context, the
so called Chern conjecture, was refuted: Shankar [1998] constructed
a positively curved manifold with a non cyclic abelian fundamental group.

As we will discuss in the last section the known methods for constructing
nonnegatively curved manifolds are somewhat limited. The most
important tools are the O'Neill formulas which imply  that the base of a
Riemannian submersion has nonnegative  (positive) sectional curvature if the
total space has. We recall that a smooth surjective map $\sigma\colon M\rightarrow B$ 
between two Riemannian manifolds is called a Riemannian submersion 
if the dual $\sigma_*^{ad}\colon T_{\sigma(p)}B\rightarrow T_pM$ 
of the differential of $\sigma$ is length preserving for all $p\in M$.
Apart from taking products, the only other method is
a special glueing technique, which was used by Cheeger, and more recently
by Grove and Ziller to construct quite a few interesting examples
of nonnegatively curved manifolds. 

By comparing with the class of known positively curved manifolds, 
the nonnegatively curved manifolds form a huge class. In fact in dimensions
above $24$ all known simply connected compact  positively curved manifolds 
are diffeomorphic to rank $1$ symmetric spaces. Due to work of the author
the situation is somewhat better in the class of known examples of manifolds
with positive curvature on open dense sets, see section~\ref{sec: pos}.

Given the drastic difference in the number of known examples, it is somewhat
painful that the only known obstructions on positively curved compact manifolds, 
which do not remain valid for the nonnegatively curved manifolds, are the above
quoted results of Synge and Bonnet Myers on the fundamental groups.

Since the list of general  structure results is not far from being complete by now, 
the reader might ask why a survey on such a subject is necessary. The reason
is that there are a lot of other beautiful theorems in the subject including
structure results, but they usually need additional assumptions.

We have subdivided the paper in five sections. Section~\ref{sec: sphere} is on sphere theorems and related rigidity results, some notes
 on very recent significant developments 
were added in proof and can be found in section~\ref{sec: added}.
 In section~\ref{sec: cpt nonneg},
we survey results on compact nonnegatively curved manifolds, and in  
section~\ref{sec: open nonneg}, results on open nonnegatively curved manifolds. 
Then follows a section on compact positively curved manifolds with symmetry, since 
this was a particularly active area in recent years. 
Although we pose problems and conjectures throughout the paper we close the paper
with a section on open problems.

We do not have the ambition to be complete 
or to sketch  all the significant historical developments 
that eventually led to the stated results. Instead we will usually only quote a few things according to personal taste.

\section{Sphere theorems and related rigidity results.}\label{sec: sphere} 

A lot of techniques in the subject were developed or used in 
connection with proving sphere theorems. In this section 
we survey some of these results. 
We recall Toponogov's triangle comparison theorem. 
Let $M$ be a complete manifold with sectional curvature $K\ge \kappa$ 
and consider a geodesic triangle $\Delta$ in $M$ 
consisting of minimal geodesics with length $a,b,c\in \R$. 
Then there exists a triangle in the
$2$-dimensional complete surface $M^2_{\kappa}$ of constant curvature $\kappa$
with side length $a,b,c$ and the angles in the comparison triangle 
bound the corresponding angles in $\Delta$ from below.
\subsection{ Topological sphere theorems.}
We start with the classical sphere theorem of Berger and Klingenberg.

\begin{thm}[Quarter pinched sphere theorem] Let $M$ be a complete simply connected manifold with sectional curvature 
$1/4 < K \le 1$. Then $M$ is homeomorphic to the sphere.
\end{thm}

The proof has two parts. The first part is to show that the injectivity radius 
of $M$ is at least $\pi/2$. This is elementary in even dimensions. 
In fact by Synge's Theorem any even dimensional oriented manifold with curvature 
$0<K\le 1$ has injectivity radius $\ge \pi$.  In odd dimensions the result is due 
to Klingenberg and relies on a more delicate Morse theory argument on the loop space.

The second part of the proof is due to Berger. He showed that any manifold 
with injectivity radius $\ge \pi/2$ and curvature $>1$ is homeomorphic to a sphere. 
In fact by applying Toponogov's theorem to two points of maximal distance, he showed   
that the manifold can be covered by two balls, which are via the 
exponential map diffeomorphic to balls in the Euclidean space.

Grove and Shiohama [1977] gave a significant improvement of Berger's theorem, by replacing the 
lower injectivity radius bound by a lower diameter bound.

\begin{thm}[Diameter sphere theorem] Any manifold with sectional curvature $\ge 1$ and diameter $>\pi/2$ is homeomorphic 
to a sphere. 
\end{thm}
 
More important than the theorem was the fact the proof introduced a new concept: 
critical points of distance functions. A point $q$ is critical with respect 
to the  distance function $d(p,\cdot)$ if the set of initial vectors 
of minimal geodesics from $q$ to $p$ intersect each closed half  space of $T_qM$. 
If the point $q$ is not critical it is not hard to see that 
there is a gradient like vectorfield $X$ in a neighborhood of $q$. 
A vectorfield is said to be gradient like if for each integral curve $c$ of $X$
the map $t\mapsto d(p,c(t))$ is a monotonously increasing bilipschitz map
onto its image. An elementary yet important observation is that 
local gradient like vectorfields can be glued together using a partition of 
unity.

\begin{proof}[Proof of the diameter sphere theorem.]
 We may scale the manifold such that its diameter is $\pi/2$ 
and the curvature is strictly $>1$.
Choose two points $p$, $q$ of maximal distance $\pi/2$, 
and let $z$ be an arbitrary third point. 
Consider the spherical comparison triangle $(\tilde p,\tq,\tz)$. 
We do know that the side length of $(\tp,\tz)$ and $(\tq,\tz)$ 
are less or equal to $\pi/2$ whereas $d_{\Sph^2}(\tp,\tq)=\pi/2$. 
This implies that the angle of the triangle at $\tz$ is $\ge \pi/2$. 
By Toponogov's theorem any minimal geodesic 
triangle with corners $p,q,z$ in $M$ has an 
angle strictly larger than $\pi/2$  based at $z$.
This in turn implies that the distance function $d(p,\cdot)$ has no critical points 
in $M\setminus \{p,q\}$. 
 Thus there is a gradient like vectorfield $X$ on $M\setminus \{p,q\}$.
Furthermore without loss of generality 
 $X$ is given on $B_r(p)\setminus \{p\}$ 
by the actual gradient of the distance function $d(p,\cdot)$, 
where $r$ is smaller than the injectivity radius. 
We may also assume 
$\|X(z)\|\le d(q,z)^2$ for all $z\in M\setminus \{p,q\}$. Then the flow $\Phi$ of $X$ exists 
for all future times and 
we can define a diffeomorphism \[
\psi\colon T_pM\rightarrow M\setminus \{q\}
\]
as follows: for a unit vector
$v\in T_pM$ and a nonnegative number $t$ put 
$\psi(t\cdot v)=\exp(tv)$ if
$t\in [0,r]$ and $\psi(t\cdot v)=\Phi_{t-r}(\exp(rv))$ if $t\ge r$. 
Clearly this implies that $M$ is homeomorphic to a sphere.
\end{proof}

There is another generalization of the sphere theorem of Berger and  Klingenberg.
A manifold is said to have positive isotropic curvature
if for all orthonormal vectors $e_1,e_2,e_3,e_4\in T_pM$ the 
curvature operator satisfies
\[
R(e_1\wedge e_2 + e_3\wedge e_4,e_1\wedge e_2 + e_3\wedge e_4)+
 R(e_1\wedge e_3 + e_4\wedge e_2,e_1\wedge e_3 + e_4\wedge e_2)>0
\]

By estimating the indices of minimal $2$ spheres in a manifold of positive isotropic curvature,
Micallef and Moore [1988] were able to show that 

\begin{thm} A simply connected compact Riemannian manifold of positive isotropic curvature 
is a homotopy sphere.
\end{thm}

A simple computation shows that pointwise strictly quarter pinched manifolds have positive 
isotropic curvature. Thus the theorem of Micallef and Moore is a generalization 
of the quarter pinched sphere theorem.
A more direct improvement of the quarter pinched sphere theorem is due to 
Abresch and Meyer [1996]. 

\begin{thm}\label{thm: AM} Let $M$ be a compact simply connected
 manifold with sectional curvature
$\tfrac{1}{4(1+10^{-6})^{2}}\le K\le 1$.
Then one of the following holds 
\begin{enumerate}
\item[$\bullet$] $M$ is homeomorphic to a sphere.
\item[$\bullet$] $n$ is even and the cohomology ring $H^*(M,\Z_2)$ is generated by 
one element.
\end{enumerate}
\end{thm}

It is a well known result in topology that the $\Z_2$ cohomology rings of spaces which are generated by 
one element are precisely given by the $\Z_2$ -cohomology rings of rank 1 symmetric spaces
$\RP^n,\CP^n,\HP^n,\CaP$ and $\Sph^n$, cf. [Zhizhou, 2002].

The proof of Theorem~\ref{thm: AM} has again two parts. Abresch and Meyer first establish 
that the injectivity radius of $M$ 
is bounded below by the conjugate radius 
which in turn is bounded below by $\pi$. 
From the diameter sphere theorem it is clear that 
without loss of generality  $\diam(M,g)\le \pi (1+10^{-6})$.
They then establish the horse shoe inequality, which was conjectured by Berger: 
for $p\in M$ and any unit vector $v\in T_pM$ 
one has \[
d(\exp(\pi v),\exp(-\pi v))< \pi.
\] 
In particular $\exp(\pi v)$ and $\exp(-\pi v)$ 
can be joined by a unique minimal geodesic. 
Once the horse shoe inequality is established it is easy to see 
that there is a smooth map $f\colon \RP^n\rightarrow M^n$ 
such that in odd dimensions the integral degree is $1$ 
and in even dimensions the $\Z_2$-degree is $1$. 
The theorem then follows by a straightforward cohomology computation. 

The horse shoe inequality relies on a mixed Jacobi field 
estimate. 
We only state the problem here 
in a very rough form. Let $c$ be a normal geodesic in $M$ 
and $J$ a Jacobifield with $J(0)=0$. 
Suppose that at time $t_0=\tfrac{2\pi}{3}$ the value  
$\|J(t_0)\|$  is quite a bit smaller than one would 
expect by Rauch's comparison from the lower curvature bound. 
Can one say that $\|J(t)\|$ is also quite a bit smaller than in Rauch's 
comparison
for $t\ge [t_0,\pi]$? 
Abresch and Meyer gave an affirmative answer. 
If one wants to improve the pinching constant 
one certainly needs to improve their Jacobifield estimate.

\subsection{Differentiable sphere theorems.}

It is not known whether there are exotic spheres with positive 
sectional curvature. 
A closely related question is whether one can improve
in any (or all) of the above mentioned topological sphere theorems
the conclusion from homeomorphic to being diffeomorphic to 
a sphere. In other words, can one turn the 
topological sphere theorems into differentiable sphere theorems. 
In each case this is an open question. 
However, there are quite a few differentiable sphere 
theorems, which hold under stronger assumptions.

The first differentiable sphere theorem was established  in his thesis 
by Gromoll. He had a pinching condition $\delta(n) < K \le 1$
 but his pinching constant $\delta(n)$ 
depended upon the dimension, i.e. $\delta(n)\to 1$ for $n\to \infty$. 

Sugimoto and Shiohama [1971] established the first bound which was independent of 
the dimension with $\delta=0.87$.  
In a series of papers Grove, Im Hof, Karcher and Ruh obtained the following 
result

\begin{thm} There is a decreasing 
sequence of numbers $\delta(n)$ with $\lim_{n\to \infty} \delta(n)=0.68$
such that any simply connected manifold $(M,g)$ with 
$\delta(n)<K\le 1$ is diffeomorphic to the sphere $\Sph^n$. 
Furthermore the diffeomorphism may be chosen such that 
the natural action  $\Iso(M,g)$ on $M$ corresponds under $f$ 
to a linear 
action on $\Sph^n$.
\end{thm}

If one does not insist on an equivariant diffeomorphisms, then 
the pinching constant can be improved somewhat. Suyama [1995] showed 
that a simply connected manifold with $0.654<K\le 1$ 
is diffeomorphic to the sphere.

The work of Weiss [1993] goes in a different direction. 
 He uses the fact that a quarter pinched sphere $M^n$ 
 has Morse perfection $n$. 
A topological sphere $M^n$ is said that to have Morse perfection 
$\ge k$ if there  
is a smooth map $\Psi\colon \Sph^k\rightarrow \C^\infty(M,\R)$ 
satisfying $\Psi(-p)=-\Psi(p)$, 
and for each $p\in \Sph^k$ the function $\Psi(p)$ 
is a Morse function with precisely two critical points. 
It is not hard to see that a quarter pinched sphere has Morse perfection 
$n$. Weiss used this to rule out quite a few of the exotic spheres 
by showing that their Morse perfection is $<n$. 
He showed that in dimensions $n=4m-1$ any exotic 
sphere bounding a parallelizable manifold 
has odd order in the group of exotic spheres unless 
the Morse perfection $\le n-2$.

By Hitchin there are also exotic spheres with a non-vanishing
 $\alpha$-invariant, and thus these spheres do not even admit 
metrics with positive scalar curvature, see the survey 
of Jonathan Rosenberg.

Similar to the quarter pinched sphere theorem, 
one can also strengthen the assumptions in the diameter 
sphere theorem in order to get a differentiable 
sphere theorem. This was carried out by Grove and Wilhelm [1997]. 
\begin{thm} Let $M$ be an $n$-manifold with sectional curvature $\ge 1$ 
containing  $(n-2)$-points with pairwise distance $>\pi/2$. Then $M$ is diffeomorphic 
to a sphere. 
\end{thm}

If one has only $k$ points with pairwise distance $>\pi/2$, then 
Grove and Wilhelm obtain restrictions on the differentiable structure 
of $M$. 

With a slight variation of the proof of Grove and Wilhelm one 
can actually get a slightly better result.
Let $M$ be an inner metric space. 
We say that $M$
has a weak $2$-nd packing radius $\ge r$ 
if $\diam(M)\ge r$. We say it has a 
weak $k$-th packing radius $\ge r$ 
if there is a point $p\in M$ 
such that $\partial B_{r}(p)$ is connected and 
 endowed with its inner metric 
has weak $(k-1)$-th packing radius $\ge r$.

\begin{thm}\label{thm: alex deformation}\label{thm: alex}
Let $(M,g)$ be an $n$-manifold with sectional curvature $\ge 1$ 
and weak $(k+1)$-th packing radius $>\pi/2$.
 Then there is a family of metrics $g_t$  $(t\in [0,1)$ with
sectional curvature $\ge 1$ and $g_0=g$ such that $(M,g_t)$ converges 
for $t\to 1$ to an $n$-dimensional Alexandrov space $A$ satisfying: 
If $k\ge n$, then $A$ is isometric to the standard sphere. 
If $k<n$, then $A$ is given by the $k$-th iterated 
suspension $\Sigma^k A'$ of an $n-k$-dimensional Alexandrov space $A'$.
\end{thm}

\begin{cor} Let $\eps>0$. A manifold with sectional curvature $\ge 1$ 
and diameter $>\pi/2$ also admits
 a metric with sectional 
curvature $\ge 1$ and diameter $>\pi -\eps$.   
\end{cor}
As in the paper of Grove and Wilhelm,
one can show in the situation of Theorem~\ref{thm: alex deformation}
that there is a sequence 
of positively metrics  $\tg_i$ on the standard sphere with 
curvature $\ge 1$ such that $(\Sph^n,\tg_i)_{i\in\N}$ converges to 
$A$ as well. In particular, Grove and Wilhelm showed
 that an affirmative answer 
to the following question would imply the 
differentiable diameter sphere theorem. 

\begin{question}[Smooth stability conjecture] Suppose a sequence of compact $n$-manifolds $(M_k,g_k)$ 
with curvature $\ge -1$ converges in the Gromov Hausdorff topology
 to an $n$-dimensional compact Alexandrov space $A$. Does this imply 
that for all large $k_1$ and $k_2$ the manifolds $M_{k_1}$ and $M_{k_2}$ are diffeomorphic?   
\end{question}

By Perelman's stability theorem it is known that $M_{k_1}$ and 
$M_{k_2}$ are homeomorphic for all large $k_1$ and $k_2$, see 
the article of Vitali Kapovitch in this volume.  

\begin{proof}[Sketch of the proof of Theorem~\ref{thm: alex deformation}.]
Let $p,q\in M$ be  points such that $d(p,q)>\pi/2+\eps$ for some $\eps >0$.
We claim that we can find a 
continuous family of metrics with $g_0=g$ and $K_t\ge 1$ 
such that $(M,g_t)$ converges for $t\to 1$ 
to the suspension of $\partial B_{\pi/2}(p)$.

We consider the suspension $X$ of $M$, i.e.,
$X= [-\pi/2,\pi/2]\times M/\sim$ where the equivalence 
classes of $\sim$ are given by $p_+:=\{\pi/2\}\times M$, 
$p_-:=\{-\pi/2\}\times M$ and the one point sets $\{(t,p)\}$ for $|t|\neq \pi/2$.
Recall that $X$ endowed with the usual warped product metric
 is an Alexandrov space with curvature $\ge 1$.

We consider the curve $c(t)=((1-t)\pi/2,p)$ 
as a curve in $X$, $r(t):= \pi/2+\eps(1-t)$
and the ball $B_{r(t)}(c(t))\subset X$. 
Put $N_t:=\partial B_{r(t)}(c(t))$. 
Since $X\setminus B_{r(t)}(c(t))$ is strictly convex 
and $N_t$ is contained in the Riemannian manifold $X\setminus \{p_\pm\}$ for all $t\neq 1$, it follows 
that $N_t$ is an Alexandrov space with curvature $\ge 1$ for all 
$t\in [0,1]$. Clearly $N_0$ is up to a small scaling factor isometric to $M$. 
Moreover $N_1$ is isometric to the suspension of 
$\partial B_{\pi/2}(p)\subset M$.

Using that $N_t$ is strictly  convex in the Riemannian manifold 
$X\setminus \{p_{\pm}\}$ for $t\in [0,1)$, it follows that the family $N_t$  
can be approximated by a family of strictly convex smooth 
submanifolds $\tilde{N}_t\subset X\setminus \{p_{\pm}\}$, $t\in [0,1)$. 
Furthermore, one may assume that $\lim_{t\to 1} N_t=N_1=\lim_{t\to 1} \tilde{N}_t$.

We found a family of metrics $g_t$ of curvature $>1$ 
such that $(M,g_t)$ converges to the suspension of 
 $\partial B_{\pi/2}(p)$. We may assume that  $\partial B_{\pi/2}(p)$
has weak $k$-th packing radius $>\pi/2$ and $k\ge 2$.

We now choose a curve of  points $q_t \in M$ 
converging for $t\to 1$ to a point on the equator $q_1\in \partial  B_{\pi/2}(p)$
  of the limit space such that
there is a point $q_2$ in $ \partial  B_{\pi/2}(p)$ 
whose intrinsic distance to $q_1$ is $>\pi/2$. 

We now repeat the above construction 
for all $t\in (0,1)$ with $(M,g,p)$ 
replaced by $(M,g_t,q_t)$. 
This way we get for each $t$ an one parameter family 
of smooth metrics $g(t,s)$ with $K\ge 1$ 
which converges for $s\to 1$ to the suspension of 
the boundary of $B_{\pi/2}(q_t)\subset (M,g_t)$.
It is then easy to see that one can choose 
the metrics such that they depend smoothly on $s$ and $t$.
Moreover, after a possible reparameterization of $g(s,t)$
the one parameter family $t\mapsto g(t,t)$
converges to the double suspension of 
the boundary of $B_{\pi/2}(q_1)\subset \partial B_{\pi/2}(p)$.
Clearly the theorem follows by  iterating this process.
\end{proof}


We recall that to each Riemannian manifold $(M,g)$  
and each point $p\in M$ one can assign a curvature operator
 $R\colon \Lambda_2T_pM \rightarrow \Lambda_2T_pM$. 
We call the operator $2$-positive if the sum of the smallest two 
eigenvalues is positive. It is known that manifolds 
with $2$-positive curvature operator have positive isotropic curvature.

\begin{thm}\label{thm: boehm wilking}\label{thm: 2 pos} Let $(M,g)$ be a compact manifold with 
$2$-positive curvature operator. Then the normalized Ricci flow evolves 
$g$ to a limit metric of constant sectional curvature.
\end{thm}

In dimension $3$ the theorem is due to Hamilton [1982]. 
Hamilton [1986] also showed that the theorem holds for $4$-manifolds with   
positive curvature operator. This was extended by Chen to $4$-manifolds with 
$2$-positive curvature operator. 
In dimension $2$ it was shown by Hamilton and Chow that for any surface 
the normalized  Ricci flow converges to limit metric of constant curvature. 
In dimensions above $4$ the theorem is due to B\"ohm and Wilking [2006]. 
For $n\ge 3$ the proof solely relies on the maximum principle and works 
more generally in the category of orbifolds. 

We recall that a family of metrics $g_t$ on $M$ 
is said to be a solution of the Ricci flow if 
\[
\tfrac{\partial}{\partial t} g_t= -2\Ric_t
\]
Hamilton showed that if one represents 
the curvature operator $R$ with respect to suitable moving  
orthonormal frames, then  
\[
\tfrac{\partial}{\partial t} R= \Delta R+ 2(R^2+R^\#)
\]
where $R^\#=\ad\circ R\wedge R\circ \ad^*$, $\ad \colon \Lambda_2\so(T_pM)\rightarrow \so(T_pM)$ 
is the adjoint representation and where we have identified  
$\Lambda_2 T_pM$ with the Lie algebra $\so(T_pM)$. 
Hamilton's maximum principle allows to deduce certain dynamical properties 
of the PDE from dynamical properties of the ODE
\[
\tfrac{d}{d t} R= R^2+R^\#.
\]

\begin{proof}[Sketch of the proof of Theorem~\ref{thm: 2 pos}.]
We let $S_B^2\bigl(\so(n)\bigr)$ denote the vectorspace of algebraic curvature 
operators satisfying the Bianchi identity. 

 We call a continuous family $C(s)_{s\in [0,1)}\subset
S_B^2(\so(n))$  of  closed convex $\Or(n)$-invariant cones of full
dimension a pinching family, if
\begin{enumerate}
\item each $\Rc\in C(s)\setminus \{0\}$ has positive scalar curvature,

\item $\Rc^2+\Rc^\#$ is contained in the interior of the tangent
      cone of $C(s)$ at $\Rc$ for all $\Rc \in C(s)\setminus \{0\}$
      and all $s\in (0,1)$,
\item $C(s)$ converges in the pointed Hausdorff topology to the
      one-dimensional cone $\R^+ I$ as $s\to 1$.
\end{enumerate}

The argument in [B\"ohm and Wilking, 2006] has two parts. 
One part is a general argument showing for any pinching 
family $C(s)$ ($s\in [0,1)$) that on any compact manifold 
$(M,g)$ for which the curvature operator is contained 
in the interior of $C(0)$ at every point the normalized 
Ricci flow evolves $g$ to a constant curvature limit metric. 
In the proof of this result one first constructs 
to such a pinching family a pinching set in the sense 
Hamilton which in turn gives the convergence result.

The harder problem is actually to construct a pinching family 
with $C(0)$ being the cone of $2$-nonnegative curvature operators. 
Here a  new tool is established. It 
is a formula that describes how this ordinary differential equation 
$R'=R^2+R^\#$ changes under  $\Or(n)$-equivariant linear transformations.
  By chance the transformation law 
is a lot simpler than for a generic $\Or(n)$- invariant 
quadratic expression. 
The transformation law often allows to construct new  
ODE-invariant curvature cones as the image 
of a given 
invariant curvature cone under suitable equivariant linear transformation 
$l\colon S_B^2\bigl(\so(n)\bigr)\rightarrow S_B^2\bigl(\so(n)\bigr)$.
This in turn is used to establish the existence of a pinching family.
\end{proof}

\subsection{Related rigidity results} 

We first mention the diameter rigidity theorem of Gromoll and Grove [1987]
\begin{thm}[Diameter rigidity]
 Let $(M,g)$ be a compact manifold with sectional curvature $K\ge 1$
and diameter $\ge \pi/2$. Then one of the following holds:
\begin{enumerate}
\item[a)] $M$ is homeomorphic to a sphere.
\item[b)] $M$ is locally isometric to a rank one symmetric space.\\[-1ex]
\end{enumerate}
\end{thm}
The original theorem allowed a potential exceptional case 
\begin{enumerate}
\item[$\bullet$] $M$ has the cohomology ring of the Cayley plane, but is not isometric 
to the Cayley plane.  
\end{enumerate}
This case was ruled out much later by the author, see [Wilking, 2001].

The proof of the diameter rigidity theorem is closely linked to the rigidity of Hopf 
fibrations which was established by Gromoll and Grove [1988] as well
\begin{thm}[Rigidity of Hopf fibrations]\label{thm: hopf}
 Let $\sigma\colon \Sph^n \rightarrow B$ be a Riemannian submersion 
with connected fibers. Then $\sigma$ is metrically congruent to a Hopf fibration. 
In particular the fibers are totally geodesic and $B$ is rank one symmetric space.
\end{thm}

Similarly to the previous theorem, the original theorem allowed for a possible exception, 
Grove and Gromoll assumed in addition $(n,\dim B)\neq (15,8)$. 
Using very different methods, the rigidity  of this special case 
was proved by the author in [Wilking, 2001]. This in turn ruled out the exceptional case 
in the diameter rigidity theorem as well. 

\begin{proof}[Sketch of the proof of the diameter rigidity theorem.]
The proof of the diameter rigidity theorem is the most beautiful 
rigidity argument in positive curvature. 
One assumes that the manifold is not homeomorphic to a sphere. 
Let $p$ be a point with $N_2:=\partial B_{\pi/2}(p)\neq \emptyset$. 
One defines $N_1=\partial B_{\pi/2}(N_2)$ as the boundary of the distance 
tube $B_{\pi/2}(N_2)$ around $N_2$.
It then requires some work to see that $N_1$ and $N_2$ are totally 
geodesic submanifolds without boundary 
satisfying $N_2=\partial B_{\pi/2}(N_1)$.

Not both manifolds can be points, since otherwise one 
can show that $M$ is homeomorphic to a sphere. 
If one endows the unit normal bundle $\nu^1(N_i)$ with its natural 
connection metric, then Grove and Gromoll show in a next step 
that the map  $\sigma_i\colon \nu^1(N_i)\rightarrow N_j$, 
$v\mapsto \exp(\pi/2 v)$ is a Riemannian submersion, $\{i,j\}=\{1,2\}$.
 Furthermore $\sigma_i$ restricts to a Riemannian submersion 
$\nu^1_q(N_i)\rightarrow N_j$ for all $q\in N_i$.
  
In the simply connected case one shows that $N_i$ is simply connected 
as well, $i=1,2$. By the rigidity of submersions defined on Euclidean 
spheres (Theorem~\ref{thm: hopf}) we deduce that $N_i$ is 
either a point or a rank one symmetric space with diameter $\pi/2$.  
Going back to the definition of  $N_1$, it is then 
easy to see that $N_1=\{p\}$. 
Using that $\sigma_1\colon \Sph^{n-1}\rightarrow N_2$ 
is submersion with totally geodesic fibers, 
one can show that the pull back metric $\exp^*_p g$ 
 on $B_{\pi/2}(0)\subset T_pM$ is determined by $\sigma_1$. 
Thus $M$ is isometric to a rank one symmetric space. 

In the non simply connected case one can show that either 
the universal cover is not a sphere and thereby symmetric or 
 $\dim(N_1)+\dim(N_2)=n-1$. In the latter case it is not hard to verify 
that $M$ has constant curvature one. 
\end{proof}
 
Since the proof of the differentiable sphere theorem 
for manifolds with $2$-positive curvature follows from a Ricci flow 
argument it is of course not surprising that it has 
a rigidity version as well.

\begin{thm}\label{thm: rigidity ricci flow}\label{thm: 2 nonnon}
 A simply connected compact manifold with $2$-nonnegative curvature 
operator satisfies one of the following statements.  
\begin{enumerate}
\item[$\bullet$] The normalized Ricci flow evolves the metric 
to a limit metric which is up to scaling is isometric to $\Sph^n$  or 
 $\CP^{n/2}$.
\item[$\bullet$] $M$ is isometric to an irreducible symmetric space. 
\item[$\bullet$] $M$ is isometric to nontrivial Riemannian product.
\end{enumerate}
\end{thm} 
Of course in the last case the factors of $M$ have nonnegative 
curvature operators. By Theorem~\ref{thm: deforming}
$(M,g)$ 
admits a possibly different metric $g_1$ such that $(M,g_1)$ 
is locally isometric to $(M,g)$ and $(M,g_1)$ is finitely covered 
by a Riemannian product $T^d\times M'$ where $M'$ is simply connected and compact.
This effectively gives a reduction to the simply connected case.

The theorem has many names attached to it. Of course 
Theorem~\ref{thm: boehm wilking} 
(Hamilton [1982,1986], B\"ohm and Wilking [2006])  enters as the 'generic'
case. This in turn was used by Ni and Wu [2006] to reduce the problem 
from $2$-nonnegative curvature operators to nonnegative curvature 
operators.
One has to mention  Gallot and Meyer's [1975] investigation of manifolds 
with nonnegative curvature operator using the Bochner technique. 
Berger's classification of holonomy groups, 
 as well as
 Mori's [1979], Siu and Yau's [1980] solution of the Frankel conjecture 
are key tools. Based on this
Chen and Tian [2006] proved convergence of 
the Ricci flow for compact K\"ahler manifolds with positive bisectional curvature.

\begin{proof}[Sketch of a proof of Theorem~\ref{thm: 2 nonnon}.]  
Consider first the case that the curvature operator of $M$ is not 
nonnegative. We claim that then the Ricci flow immediately 
evolves 
$g$ to a metric with $2$-positive curvature operator. 

We consider a short time solution $g(t)$  of  the Ricci flow and 
let $f\colon [0,\eps)\times M\rightarrow \R$, denote the function 
which assigns to $(t,p)$ the sum of the lowest two eigenvalues  
of the curvature operator of $(M,g(t))$ at $p$. 
We first want to show that $f(t,\cdot)$ is positive somewhere for small $t>0$.
We may assume that $f(0,p)=0$ for all $p$. 
It is straightforward to check that $f$ satisfies 
\[
\tfrac{\partial f}{\partial t}_{|t=0_+}(0,p)\ge q(R):=\tfrac{\partial}{\partial t}_{|t=0_+} (\lambda_1+\lambda_2)\bigr(R+t(R^2+R^\#)\bigl).
\]
From the invariance of $2$-nonnegative curvature operators 
 it is known that $q(R)\ge 0$.
In fact a detailed analysis of the proof shows that 
$q(R)\ge 2 (\lambda_1(R))^2$. 
In the present situation we deduce by a first order argument that $f(t,p)$ 
becomes positive somewhere for small $t>0$. Now it is not hard to establish 
a strong maximum principle that shows that $f(t,\cdot)$ is everywhere positive for small 
$t>0$,
see Ni and Wu [2006].
 In other words $(M,g_t)$ has $2$-positive curvature operator for $t>0$ 
and the result follows from Theorem~\ref{thm: boehm wilking}.

We are left with the case that the curvature operator of $(M,g)$ is nonnegative.
Essentially this case was already treated by Gallot and Meyer using the Bochner technique,
 see [Petersen, 2006]. 
We present a slightly different argument following Chow and Yang (1989).
Using Hamilton's [1986] strong maximum principle one deduces that for  $t>0$ the curvature 
operator of  
$(M,g_t)$ has constant rank and that the kernel is parallel. Thus either 
$\Rc_t$ is positive or the holonomy is non generic. We may assume that $M$ does not split 
as a product. Hence without loss of generality $M$ is irreducible with non generic holonomy. 
Since $(M,g_t)$ clearly has positive scalar curvature Berger's classification of holonomy groups
 implies 
that Hol$(M)\cong \gU(n/2),\Sp(1)\Sp(n/4)$ unless $(M,g)$ is a symmetric space.  In the  
case of Hol$(M)\cong \Sp(1)\Sp(n)$ we can employ another theorem of Berger [1966]
to see that $M$ is up to scaling isometric to 
$\HP^{n/4}$, since in our case the sectional curvature of $(M,g_t)$ is positive. 
 In the remaining case Hol$(M)=\gU(n/2)$ 
it follows that $M$ is K\"ahler and $(M,g_t)$ has positive (bi-)sectional 
curvature. By Mori [1979] and Siu and Yau's [1980] solution of the Frankel conjecture 
$M$ is biholomorphic to $\CP^{n/2}$. 
In particular, $M$ admits a K\"ahler Einstein metric. 
 Due to work 
of Chen and Tian [2006] it follows, that the normalized Ricci flow on $M$ 
converges to the Fubini study metric which completes the proof. 
\end{proof}

\section{Compact nonnegatively curved manifolds}\label{sec: cpt nonneg}

The most fundamental obstruction to this date is Gromov's Betti number theorem.
\begin{thm}[Gromov, 1981]\label{thm: gromov} Let $M^n$ be an $n$-dimensional complete manifold
with nonnegative sectional curvature, and let $\fF$ be a field. Then the total Betti
number satisfies 
 \[
 b(M,\fF):=\sum_{i=0}^n b_i(M,\fF)\;\le\; 10^{10n^4}.
 \] 
\end{thm}
 Gromov's original bound on the total Betti number was depending double exponentially
 on the dimension. The improvement is due to Abresch [1987]. However, this  bound is not optimal
 either. In fact Gromov posed the problem  whether the best possible bound is $2^n$, 
 the total Betti number of the $n$-dimensional torus. 
The statement is particularly striking since the nonnegatively curved 
manifolds in a fixed dimension $\ge 7$ have infinitely many homology types with respect 
to integer coefficients. More generally Gromov gave explicit estimates 
for the total Betti numbers of compact $n$-manifolds with curvature $\ge -1$ and 
diameter $\le D$.
The proof is an ingenious combination of Toponogov's theorem and critical point 
theory.
\begin{proof}[Sketch of the proof of Theorem~\ref{thm: gromov}.]
 The most surprising part in the proof is  a definition: 
Gromov assigns to every ball $B_r(p)\subset M$ a finite number 
called the corank of the ball. It is defined as the maximum over all 
$k$ such that 
for all $q\in B_{2r}(p)$ there are  points $q_1,\ldots,q_k$ 
with 
\[
d(q,q_1)\ge 2^{n+3}r,\,\, d(q,q_{i+1})\ge 2^n d(p,q_i)
\] 
and $q_i$ is a critical point of the distance function of $q$ in the sense 
of Grove and Shiohama. One can show as follows that 
the corank of a ball is at most $2n$: Choose a minimal
geodesic $c_{ij}$ from $q_i$ to $q_j$, $i<j$ 
and minimal geodesic $c_i$ from $q$ to $q_i$, $i=1,\ldots,k$.
Since $q_i$ is a critical point we can find a possibly different minimal 
geodesic $\tilde c_i$ from $q$ to $q_i$ such that the angle  
of the triangle $(\tilde c_i,c_j,c_{ij})$ based at $q_i$ is $\le \pi/2$. 
Therefore $L(c_j)^2\le L(c_{ij})^2+L(c_i)^2$. 
Applying Toponogov's theorem to the triangle $(c_i,c_j,c_{ij})$ gives that 
the angle $\varphi_{ij}$ between $c_i$ and $c_j$ 
satisfies $\tan(\varphi_{ij})\ge 2^n$. Thus $\varphi_{ij}\ge \pi/2-2^{-n}$. 
 The upper bound on $k$ now follows from an Euclidean sphere packing argument
in $T_qM$.

By reverse induction on the corank, one establishes an estimate
 for the content of a 
ball  $\cont(B_r(p))$ which is defined as
the dimension of the image of $H_*(B_r(p))$ in $H_*(B_{5r}(p))$. 
A ball $B_r(p)$ with maximal corank is necessarily contractible in $B_{5r}(p)$ since for some
$q\in B_{2r}(p)$ the distance function of $q$ 
 has no critical points in $B_{8r}(q)\setminus\{q\}$.
This establishes the induction base.
 It is immediate from the definition that 
 $\corank(B_{\rho}(q)) \ge \corank(B_{r}(p))$ 
for all $q\in B_{3r/2}(p)$ and all $\rho \le r/4$.
 In the induction step one  distinguishes between two cases. 

In the first case, one assumes that $\corank(B_{\rho}(q)) > \corank(B_{r}(p))$ 
for all $q\in B_r(p)$ and $\rho:=\tfrac{r}{4^n}$. 
Using the Bishop Gromov inequality it is easy to find 
a covering of $B_r(q)$ with at most $4^{n(n+2)}$ balls of radius 
$\rho$. By the induction hypothesis the balls $B_\rho(q)$
have a bounded content. Using a rather involved nested covering 
argument one can give an explicit estimate of the content of $B_r(p)$.

In the remaining case there is one point $q\in B_{r}(p)$ such that 
$\corank(B_{\rho}(q))=\corank(B_r(p))$ with $\rho= \tfrac{r}{4^n}$. 
Thus for some point $x\in B_{2\rho}(q)$ there is no critical point 
of the distance function of $x$ in $B_{8r}(x)\setminus B_{2^{-n+3}r}(x)$.
This implies that one can homotop $B_r(p)$ to a subset of $B_{r/4}(x)$ in $B_{5r}(p)$.
  From this it is not hard 
to deduce that $\cont(B_{r/4}(x))\ge \cont(B_r(p))$. 
We have seen above $\corank(B_{r/4}(x))\ge \corank(B_r(p))$. 
One can now apply the same argument again with $B_r(p)$ replaced by 
$B_{r/4}(x)$. Since small balls are contractible,
 the process has to stop after finitely many 
steps unless possibly $\cont(B_r(p))=1$.
\end{proof}

{\bf Fundamental groups.} 
Fundamental groups of 
nonnegatively curved manifolds are rather well understood. 
On the other hand, 
the known results are essentially the same as 
for compact manifolds with nonnegative Ricci curvature.
 In fact there is a general belief that the general structure 
results for fundamental groups should coincide for the two classes.
One of the main tools in this context is the splitting theorem 
 of Toponogov, resp. the splitting theorem of Cheeger and Gromoll [1971]. 
Recall that a line is a normal geodesic $c\colon \R\rightarrow (M,g)$
satisfying $d(c(t),c(s))=|t-s|$ for all $t,s\in\R$. 
By Cheeger and Gromoll's splitting theorem 
complete manifolds of nonnegative Ricci curvature 
split as products $\R\times M'$ provided they contain lines. 
In the special case of nonnegative sectional curvature, the result is 
due to  
Toponogov.

By the work of Cheeger and Gromoll [1971], 
the splitting 
theorem implies that a nonnegatively curved manifold $M$
is isometric to $\R^k\times B$ where $B$ has a compact isometry group. 
The same results holds for the universal cover 
of a compact manifold $M$ of nonnegative Ricci curvature. 
As a consequence they deduced  
that the fundamental group of $M$ is virtually abelian, i.e., it contains an abelian
subgroup of finite index. Moreover one can show 

\begin{thm}\label{thm: deforming} Let $(M,g)$ be a compact manifold 
of nonnegative Ricci curvature or an open manifold of 
nonnegative sectional curvature. Then there is a 
family of complete metrics $g_t$ on $M$ with 
$g_0=g$, $(M,g_t)$ is locally isometric to $(M,g)$ 
for all $t$ and $(M,g_1)$ 
is finitely covered by a Riemannian product 
$T^d\times M'$, where $M'$ is simply connected and 
$T^d$ is a flat torus. 
\end{thm}

The theorem is due to author [2000] but is based 
on a slightly weaker version of Cheeger and Gromoll [1971]. 
Moreover, it was shown in [Wilking, 2000] that
any finitely generated virtually abelian fundamental group occurs in some dimension
as the fundamental group of a nonnegatively curved manifold. 
However, the more interesting and challenging problem 
is what one can say about fundamental groups in a fixed dimension. 

To the best of the authors knowledge 
the only other ''effective'' result known 
for fundamental groups of nonnegative sectional 
curvature is

\begin{thm}[Gromov, 1978] The fundamental group 
of a nonnegatively curved $n$-manifold is generated 
by at most $n\cdot 2^{n}$ elements.
\end{thm}

The proof of the theorem is a simple 
application of Toponogov's theorem 
applied to the short generating system of $\pi_1(M,p)$.

Although we mentioned in the introduction 
that we will report on results which are based on collapsing techniques, 
we quote, for the sake of completeness, 
the following recent theorem 
of Kapovitch, Petrunin and Tuschmann [2005].

\begin{thm}\label{thm: kpt} For each $n$ there is a constant $C(n)$ such that the
fundamental group of any compact nonnegatively curved  $n$-manifold $(M,g)$ 
contains a nilpotent subgroup of index at most $C(n)$.
\end{thm}

The theorem remains valid for almost nonnegatively curved manifolds and it improves
a similar theorem of Fukaya and Yamaguchi from ''solvable'' to ''nilpotent''.
The proof relies on a compactness result and it remains an open problem whether one
can make the bound effective, in other words whether one can give explicit estimates
 on $C(n)$. It is also remains open whether in case of nonnegative curvature 
one can improve it from ''nilpotent'' to ''abelian''.

{\bf Other structure results.} 
%
%
By the Gauss-Bonnet formula 
a compact nonnegatively curved compact  surface is given by 
$\RP^2$, $\Sph^2$, $T^2$ or the Klein bottle. 
Due to Hamilton [1982] a
 compact $3$-manifold of nonnegative Ricci curvature 
and finite fundamental group is diffeomorphic 
to spherical space form, see Theorem~\ref{thm: 2 nonnon}.
In dimension $4$ a classification remains open. 
The best result is a theorem in Kleiner's thesis 
\begin{thm}[Kleiner]\label{thm: kleiner} Let $(M,g)$ be a nonnegatively curved 
simply connected $4$-manifold. If the isometry group is not 
finite then $M$ is homeomorphic to $\Sph^4,\CP^2$, $\Sph^2\times \Sph^2$ 
or to a connected sum $\CP^{2}\#\pm\CP^2$.
\end{thm}

The Bott conjecture (see last section) 
would imply that the theorem remains 
valid if one removes the assumption on the isometry group. 
It would be interesting to know whether one can 
improve the conclusion in Theorem~\ref{thm: kleiner}
from homeomorphic to diffeomorphic. 
Kleiner never published his thesis, but 
Searle and Yang  [1994] reproved his result. 
We present a slightly shorter proof 
which has also the advantage that it does not make 
use of a signature formula of Bott for four manifolds with Killing fields. 
This in turn implies that part of the proof carries over to 
simply connected nonnegatively curved $5$-manifolds 
with an isometric $2$-torus action. In fact using minor modifications 
it is not hard to check that the second rational Betti number 
of such a manifold is bounded above by $1$.
\begin{lem}\label{lem: trian} Let $\rho\colon \gS^1\rightarrow \Or(4)$ 
be a representation 
such that there is no trivial subrepresentation. Consider 
the induced action of $\gS^1$ on the standard sphere $\Sph^3$. 
\begin{enumerate}
\item[a)]
Any four pairwise different 
points $p_1,\ldots,p_4\in B:=\Sph^3/\gS^1$ satisfy
\[
\sum_{1\le i<  j\le 4} d(p_i,p_j) \le  2 \pi.
\]
and equality occurs if and only if 
 $B$ is isometric to 
the $2$-sphere $\Sph^2(1/2)$ of constant curvature 
$4$ and $\{p_1,p_2,p_3,p_4\}=\{\pm p,\pm q\}$.
\item[b)] The diameter of $B$ is equal to $\pi/2$.
In fact for $p\in B$ there is most one point $q\in B$ 
with $d(p,q)\ge \pi/2$. 
\end{enumerate}
\end{lem}
\begin{proof} We may assume that $\rho$ is faithful.
If the action of $\gS^1$ is given by the Hopf action, then 
$ B:=\Sph^3/\gS^1$ is the $2$-sphere $\Sph^2(1/2)$ of constant curvature 
$4$. Recall that a triangle in
$\Sph^2(1/2)$ has perimeter $\le \pi$ 
and that equality can only occur if 
two of the points on the boundary triangle have distance $\pi/2$.
Using this for all triangles
 $\{q_1,q_2,q_3\}\subset \{p_1,p_2,p_3,p_4\}$, we get 
the claimed inequality. Equality can only 
occur if the four points 
are on a great circle. A more 
detailed analysis shows that equality implies 
 $\{p_1,p_2,p_3,p_4\}=\{\pm p,\pm q\}$.

In general it is easy to construct a distance 
non-increasing homeomorphism 
\[
f\colon \Sph^2(1/2)\rightarrow B.
\]
For the proof notice that $B$ admits 
an isometric action of a circle $\gT^1$, since the centralizer of $\rho(\gS^1)$ 
in $\SO(4)$ 
acts isometrically on $B$, 
the quotient space $B/\gT^1$ is isometric to the interval $[0,\pi/2]$. 
The same holds for the quotient space $\Sph^2(1/2)/\gT^1$. 
It is now easy to see that the orbits of 
the $\gT^1$ action on $\Sph^2(1/2)$ 
are at least as long as the corresponding orbits in $B$.

Finally if the action is not given by the Hopf action, 
then we can not find four different  points 
 $p_1,p_2,p_3,p_4\in B$ with $d(p_1,p_2)=d(p_3,p_4)=\pi/2$. 
Since $f$ is distance non-increasing part b) follows as well.
\end{proof}

\begin{proof}[Proof of 
Theorem~\ref{thm: kleiner}.] 
By Freedman's classification of simply connected topological $4$-manifolds,
 it suffices to show 
that $\chi(M)\le 4$. Since the Eulercharacteristic of $M$ 
equals the Eulercharacteristic of the fixed point 
set of $\gS^1\subset\Iso(M,g)$, it suffices to estimate the latter. 

We now consider the orbit space $A^3:=M^4/\gS^1$ 
as an Alexandrov space. We first want to rule out 
that $\gS^1$ has more than four isolated fixed points. 
Suppose $p_1,\ldots,p_5$ are pairwise different 
isolated fixed points in $M$. 

We can view these points also as points in the orbit 
space $A$. 
Choose a fixed minimal normal geodesic $\gamma_{ij}\colon [0,1]\rightarrow A$
between $p_i$ and $p_j$ for $i\neq j$. 
We may assume $\gamma_{ij}$ and $\gamma_{ji}$ are equal 
up to a change of direction.

We also consider all angles $\alpha_{ijk}$ 
between $\gamma_{ij}$ and $\gamma_{ik}$ for all pairwise different  
$i,j$ and $k$.
A simple counting argument shows that there are precisely 
$30$  angles. We next prove two different estimates 
for the sum of these angles. 

For any three points in $\{p_1,p_2,p_3,p_4,p_5\}$ we get a
triangle. The sum of the angles in the triangle is $\ge \pi$,
 as $X$ is nonnegatively curved 
in the Alexandrov sense.
Therefore the sum of all $30$ angles is $\ge 10\pi$.

On the other hand we can consider for a fixed 
point $p_i$ all $6$ angles based at $p_i$.
The angles are given as the pairwise
 distances of four distinct points 
in the space of directions  
 $\Sigma_{p_i}X$.
Since $\Sigma_{p_i}X$ 
is isometric to the quotient of $\Sph^3$ by a $\gS^1$-action, 
we infer from 
Lemma~\ref{lem: trian} that the sum of these $6$ 
angles is $\le 2\pi$.
This proves that the sum of all $30$ angles is at most 
 $10\pi$.

Hence equality must hold everywhere. 
It follows that the space of directions 
at $p_i$ is given by a sphere of constant curvature 
$4$. There are precisely $10$ angles of size $\pi/2$
and for each triangle corresponding to three points in $\{p_1,\ldots,p_5\}$ 
the sum of the angles 
is $\pi$ and hence precisely one angle in such a triangle equals $\pi/2$.
 We may assume $d(p_1,p_2)=\min_{i\neq j} d(p_i,p_j)$. 
For one point $q\in \{p_3,p_4,p_5\}$ 
the triangle $(p_1,p_2,q)$ has neither an angle $\pi/2$ 
at $p_1$ nor an angle $\pi/2$ at $p_2$.  
Thus there is an angle $\pi/2$ at $q$. 
Since equality holds in Toponogov's comparison theorem 
we see 
\[
d(p_1,q)^2+d(p_2,q)^2=d(p_1,p_2)^2
\]
a contradiction since $d(p_1,p_2)$ was minimal.

Suppose next that the fixed point set $\Fix(\gS^1)$ of $\gS^1$ 
contains at least two $2$-dimensional components. 
These components form totally geodesic submanifolds 
of the Alexandrov space $A$. Since they do not intersect 
it is easy to see that $A$ is isometric to $F\times [0,l]$ 
where $F$ is a fixed point component. 
In particular $\gS^1$ has no fixed points outside the two components. 
Since each component has Eulercharacteristic $\le 2$ the result
follows.

It remains to consider the case that 
 $\gS^1$ has
precisely one $2$-dimensional fixed point component $F$.
We have to show that the  $\gS^1$-action has at most two isolated 
fixed points.  
Notice that $F$ is the boundary of the Alexandrov space 
$A$ and the distance function $h:=d(F,\cdot)\colon A\rightarrow \R$ 
is concave. 
Let $p\in A$ denote one isolated fixed point 
with minimal distance $r$ to the boundary. 
The set $h^{-1}\bigl([r,\infty[\bigr)$ is convex. 
Let $v\in \Sigma_pA$ be the initial direction 
of a minimal geodesic from $p$ to $F$. 
The tangent cone $C_p h^{-1}\bigl([r,\infty[\bigr)$ 
consist of 'vectors' which have an angle $\ge \pi/2$ 
with $v$. From Lemma~\ref{lem: trian} 
we deduce that $C_p h^{-1}\bigl([r,\infty[\bigr)$ 
is at most one dimensional. 
Thus the convex set $h^{-1}\bigl([r,\infty[\bigr)$ is either a point 
or an interval.  
By construction $h^{-1}\bigl([r,\infty[\bigr)$ contains all 
isolated fixed points of $\gS^1$. Since
for each fixed point the space of direction has diameter 
$\pi/2$, we deduce that there are 
at most two isolated fixed points.
\end{proof}

%

Gursky and LeBrun [1999] obtained strong restrictions on 
$4$-dimensional nonnegatively curved 
Einstein manifolds.

One might ask whether any nonnegatively curved compact manifold 
with finite fundamental group also admits nonnegatively curved 
  metrics with positive Ricci curvature. A partial result in direction was proved
recently.

\begin{thm}[B\"ohm and Wilking, 2005] Let $(M,g)$ be a compact nonnegatively curved  manifold with 
finite fundamental group, and let $g_t$ be a solution of the Ricci flow. Then for all small 
$t>0$, $g_t$ has positive Ricci curvature.
\end{thm}

The proof applies a dynamical version of Hamilton's maximum 
principle to a family of curvature conditions 
 lying in between nonnegative sectional curvature 
and nonnegative Ricci curvature. It then follows that $g_t$ 
has nonnegative Ricci curvature for $t\in [0,\eps]$
 with $\eps$ depending on an upper curvature bound. 
Then the theorem follows easily from a strong maximum principle. 
In the same paper it was also shown that there is no 
Ricci flow invariant curvature 
condition in between nonnegative sectional curvature 
and nonnegative Ricci curvature in dimensions above $11$.
 This in turn generalized previous results saying 
that neither nonnegative Ricci curvature nor nonnegative 
sectional curvature are invariant under the Ricci flow
in dimensions above $3$, see [Ni, 2004].

In particular, any compact nonnegatively curved manifold 
with finite fundamental group 
satisfies all obstructions coming from positive Ricci curvature.
In the simply connected case the only general known obstruction 
for positive Ricci curvature is that the manifold 
admits a metric with positive scalar curvature. By the work of 
 Gromov and Lawson and Stolz the latter
statement is equivalent to saying: Either $M$ 
is not spin or $M$ is a spin manifold with a vanishing
$\alpha$-invariant.   For more details and references 
we refer the reader 
to the surveys of Jonathan Rosenberg and Guofang Wei
published in this volume.

{\bf Grove--Ziller examples.} Recently 
 Grove and  Ziller generalized a gluing technique 
which
by the work of Cheeger [1973] was previously only known to work in the special case of
connected sums of two rank one symmetric spaces. 
Since they are discussed in more detail 
in the survey of Wolfgang Ziller we will be  brief.

\begin{thm}[Grove and Ziller, 2000]\label{thm: gz} \label{thm: grove ziller}
Let $\G$ be a compact Lie group, and let $\G/\!/\gH$ be a compact biquotient. Suppose 
there are two subgroups $\gK_{\pm}\subset \G\times \G$ such that 
$\gK_{\pm}/\gH\cong \Sph^1$ and the action of $\gK_{\pm}$ on  $\G$ is free.
Then the manifold obtained by gluing the two disc bundles associated to the two
sphere bundles $\G/\!/\gH\rightarrow \G/\!/\gK_\pm$ along their common boundary
$\G/\!/\gH$ has a metric of nonnegative sectional curvature.
\end{thm}

The stated theorem is slightly more general than the original version of Grove and Ziller,
 who considered cohomogeneity one manifolds or equivalently the case where 
all  groups $\gH, \gK_{\pm}$ act from the right on $\G$
and hence the corresponding quotients are homogeneous. Of course it would be
 interesting to know whether the generalization gives rise to any interesting
 new examples. One can actually reduce the more general statement to 
the one of Grove and Ziller as follows 

\begin{proof} We consider the manifold $M$ 
which admits a cohomogeneity one action 
 of $\G\times \G$ with principal isotropy group 
$\gH$ and singular isotropy groups $\gK_{\pm}\subset \G\times \G$. 
By Grove and Ziller this manifold has an invariant metric 
of  nonnegative sectional curvature, see the survey of 
Wolfgang Ziller for details. By assumption the diagonal 
$\Delta \G\subset \G\times \G$ acts freely on $M$. 
Clearly the manifold in the theorem is the quotient $M/\Delta \G$. 
Thus the result follows from the O'Neill formulas.
\end{proof}

\begin{thm}[Grove and Ziller] 
Any principal $\SO(n)$-bundle over $\Sph^4$ admits a cohomogeneity one action of
$\gS^3\times \SO(n)$ with singular orbits of codimension $2$.
\end{thm}

The proof uses the classification of bundles over $\Sph^4$ in terms of characteristic
classes. Grove and Ziller endow $\Sph^4$ with the unique cohomogeneity one action of $\gS^3$  with
singular orbits of codimension $2$. Then they compute for all
$\gS^3\times\SO(n)$-cohomogeneity one manifolds which are $\SO(n)$-principal bundles
over the given cohomogeneity one manifold $\Sph^4$ all characteristic classes. By
comparing the set of invariants, it follows that one gets all bundles this way. 
The details are involved.

By taking quotients of such principal bundles it follows that any sphere bundle 
over $\Sph^4$ admits a metric of nonnegative sectional curvature. This is particular striking
since $10$ of the $14$ exotic spheres in dimension $7$ can be realized as such bundles.

Grove and Ziller conjectured in their paper that any cohomogeneity one manifold 
admits an invariant nonnegatively curved metric. A partial answer was given by
Schwachh\"ofer and Tuschmann [2004] who showed that these manifolds admit metrics of
almost nonnegative sectional curvature. However, counterexamples to the Grove-Ziller
conjecture were recently found by Grove, Verdiani, Wilking and Ziller [2006]. 
The counterexamples contain all higher dimensional Kervaire spheres and therefore all
exotic spheres of cohomogeneity one. Additional counterexamples are given but to this
day it remains an open question how big the class of nonnegatively curved  cohomogeneity
one manifolds is.

\section{Open nonnegatively curved manifolds.}\label{sec: open nonneg}

Noncompact nonnegatively curved spaces often occur as blow up limits 
of sequences of manifolds  converging with 
lower curvature bound  $-1$ to a limit. 
Also recall a result of Hamilton and Ivey 
saying that for any singularity of the Ricci flow in dimension $3$ 
the corresponding blow up limit has nonnegative sectional curvature. 
This in turn was one key feature which allowed Hamilton and Perelman 
to classify the possible singularities of the Ricci flow in dimension $3$. 

By a result of Gromov [1986] any noncompact manifold admits a positively curved 
metric. However Gromov's metrics are not complete and we 
assume throughout the paper that all metrics are complete.

The structure of open manifolds of nonnegative (positive) sectional curvature is better
understood than the compact case. By a theorem of Gromoll and Meyer [1969] 
a positively curved
open manifold is diffeomorphic to the Euclidean space. For a nonnegatively curved manifold
there is the soul theorem

\begin{thm}[Cheeger and Gromoll, 1971]
 For an open nonnegatively curved manifold $M$ there is
a totally geodesic submanifold $\Sigma$ called the soul such that $M$ is diffeomorphic
to the normal bundle of $\Sigma$.
\end{thm}
\begin{proof}[Sketch of the proof]
The basic observation in the proof is that for each point $p\in M$ the function 
$f_0(q):=\lim_{r\to \infty}d(\partial B_r(p),q)-r$ is concave, proper and bounded above.
Hence the maximal level of $f_0$ is a convex closed subset $C_1$ of $M$. Cheeger and
Gromoll showed that $C_1$ is a totally geodesic compact submanifold with a possibly non-empty
and non-smooth intrinsic boundary $\partial C_1$. One can then show that if
$\partial C_1\neq \emptyset$, then the function $f_1(q)=d(\partial C_1,q)$ is concave on
$C_1$.  As before the maximal level set $C_2$ of $f_1$ is a convex subset of $M$. Since 
$\dim(C_2)<\dim(C_1)$ one can iterate the process until one arrives at a convex level set
$C_k$ without intrinsic boundary. Then $\Sigma:=C_k$ is a soul of $M$.  One can
show that the distance function $r_\Sigma:=d(\Sigma,\cdot)$ has
no critical points on $M\setminus \Sigma$ in the sense of Grove and Shiohama, 
for a definition see section~\ref{sec: sphere}. 
Thus there is a gradient like vectorfield $X$ on $M\setminus \Sigma$, with $\|X\|\le 1$.
Similarly to the proof of the diameter sphere theorem one can now construct 
a diffeomorphism $\psi\colon \nu(\Sigma)\rightarrow M$. 
\end{proof}

We emphasize that the diffeomorphism $\nu(\Sigma)\rightarrow M$ 
is in general not given by the exponential map. On the 
other hand it was shown by Guijarro [1998], that there is always 
at least one complete nonnegatively curved metric on $M$
such that this is the case.

From the soul construction it is clear that there is 
a Hausdorff continuous family $(C(s))_{s\in [0,\infty)}$ of convex compact subsets 
of $M$
such that $C(0)=\Sigma$, $C(s_1)\subset C(s_2)$ for $s_1<s_2$ and 
$\bigcup_{s\ge 0} C(s)=M$. In fact from the above sketch 
this family can be obtained by collecting all nonempty sublevels $f_i^{-1}([c,\infty[)$
of the functions $f_0,\ldots,f_{k-1}$ in one family.
Given such a family,
Sharafutdinov [1979] showed, independent of curvature assumptions, 
that there is a distance non-increasing
 retraction $P\colon M\rightarrow \Sigma$. 

\begin{thm}[Perelman, 1994]\label{thm: perelman}\label{thm: soul con}
 Let $\Sigma$ be a soul of $M$, $\nu(\Sigma)$ its normal bundle 
 and $P\colon M\rightarrow \Sigma$ a Sharafutdinov retraction. Then 
\begin{enumerate}
\item[a)] $P\circ \exp_{\nu(\Sigma)}=\pi$, where $\pi\colon\nu(\Sigma)\rightarrow \Sigma$ 
denotes the projection.
\item[b)]  Each two vectors $u\in\nu_p(\Sigma)$ and $v\in T_p\Sigma$
 are tangent to a totally geodesic immersed flat $\R^2$. 
\item[c)] $P$ is a Riemannian submersion of class $C^1$.
\end{enumerate}
\end{thm}

The theorem also confirmed the soul conjecture 
of Cheeger and Gromoll: A nonnegatively curved open manifold 
with positive sectional curvature at one point is 
diffeomorphic to $\R^n$. Although this conjecture was open for more 
than two decades, the proof of the above theorem is very short 
and just uses Rauch's comparison theorem. 

Guijarro [2000] showed that $P$ is of class $C^2$ and it was 
shown in [Wilking, 2005] that $P$ is of class $C^\infty$. 
The latter result is a consequence of another structure 
theorem on open nonnegatively curved manifolds 
whose explanation requires a bit of preparation:   
One defines a dual foliation $\folF^\#$ to 
the foliation $\folF$ given by the fiber decomposition 
$P\colon M\rightarrow \Sigma$ as follows. 
For a point $p\in M$ we define the dual leaf $\folL^\#(p)$
as the set of all points which can be connected with 
$p$ by a piecewise horizontal curve. 
We recall that a curve is called horizontal with respect 
to $P$, if it is everywhere perpendicular to the fibers of $P$. 

Because of Theorem~\ref{thm: perelman} each dual leaf 
can also be obtained as follows. Consider a vector 
$v$ in the normal bundle $\nu(\Sigma)$ of the soul. 
Let $S(v)$ denote set of all vectors in $\nu(\Sigma)$ 
 which are parallel to $v$ along some curve in $\Sigma$. 
Then $\exp(S(v))=\folL^\#(\exp(v))$. 
The structure of the  dual foliation is thus closely linked 
to the normal holonomy group of the soul.

If the normal holonomy group is transitive on the sphere, 
then the dual leaves are just given by distance spheres 
to the soul. If the holonomy group is trivial, 
then by a result of Strake [1988] and Yim [1990] $M^n$ splits
isometrically as $\Sigma^k\times (\R^{n-k},g)$ 
and the dual leaves are just given by isometric copies 
of $\Sigma$. In general the holonomy group 
is neither transitive nor trivial. In fact, by an unpublished result of the author,
 any connected subgroup of $\SO(n-k)$ can occur as the normal holonomy 
group of a simply connected soul.  

\begin{thm}[Wilking, 2005]\label{thm: dual fol} Let $M,\Sigma, P,\folF^\#$ be as above. 
\begin{enumerate}
\item[a)] Then $\folF^\#$ is a singular Riemannian foliation, i.e., 
geodesics emanating perpendicularly to dual leaves stay perpendicularly to 
dual leaves. 
\item[b)] If $u\in T_pM$ is 
horizontal with respect to $P$ and $v\in T_pM$ is perpendicular to the dual 
leaf $\folL^\#(p)$, then $u$ and $v$ are tangent to a 
totally geodesic immersed flat $\R^2$.
\end{enumerate}
\end{thm}

An analogous theorem holds 
for Riemannian submersions 
on compact nonnegatively curved manifolds.
A consequence  of the above theorem is that any non-contractible 
open nonnegatively curved manifold has an honest 
product as a metric quotient.

\begin{cor} Let $M$ be an open nonnegatively curved 
manifold and $\Sigma$ a soul of $M$. 
Then there is a noncompact Alexandrov space $A$ and 
a submetry 
\[
\sigma\colon M\rightarrow \Sigma \times A 
\]
onto the metric product $\Sigma \times A$.
The fibers of $\sigma$ are smooth compact submanifolds 
without boundary.
\end{cor}

We recall that $\sigma\colon M\rightarrow B$ is called a submetry 
if $\sigma(B_r(p))=B_r(\sigma(p))$ for all $p$ and $r$.
The space $A$ is given by the space of closures 
of dual leaves, which by Theorem~\ref{thm: dual fol} 
can be endowed with a natural quotient metric.

The main new tool used to prove these results 
is a simple and general observation which may very well be useful in 
different context as well. It allows to give what we call transversal 
Jacobi field estimates.
Let $c\colon I\rightarrow (M,g)$ be a geodesic in 
a Riemannian manifold $(M,g)$, and let
$\bbJ$  be an $(n-1)-$dimensional family of normal
Jacobi fields for which the 
 corresponding Riccati operator is self adjoint. 
Recall that the Riccati operator $L(t)$ is the endomorphism of 
$(\dc(t))^\perp$ defined by $L(t) J(t)=J'(t)$ 
for $J\in \bbJ$.
Suppose we have a vector subspace
  $\bbV\subset\bbJ$. 
Put 
\[
T^v_{c(t)}M:=\{J(t)\mid J\in \bbV\}\oplus \{J'(t)\mid J\in \bbV, J(t)=0\}.
\]
Observe that the second summand vanishes for almost every 
$t$ and that  $T^v_{c(t)}M$ depends smoothly on $t$.
We let $T^{\perp}_{c(t)}M$ denote the orthogonal 
complement of   $T^v_{c(t)}M$, and 
for $v\in T_{c(t)}M$ we define $v^{\perp}$ 
as the orthogonal projection of $v$ to $T^{\perp}_{c(t)}M$. 
If $L$ is non-singular at $t$ we put 
\[
A_t\colon T^v_{c(t)}M\rightarrow T^{\perp}_{c(t)}M, \mbox{ 
 $J(t)\mapsto J'(t)^{\perp}$ for $J\in \bbV$}.
\] 
It is easy to see that $A$ can be extended continuously on $I$. 
For a vector field $X(t)\in T^{\perp}_{c(t)}M$ we 
define $\tfrac{\nabla^{\perp}X}{\partial t}= (X'(t))^\perp$.

\begin{thm}\label{thm: jacobi}\label{thm: transverse estimate}
 Let $J\in \bbJ -   \bbV$ and put $Y(t):=J^{\perp}(t)$.
 Then $Y$ satisfies the following Jacobi equation 
\[
\tfrac{(\nabla^{\perp})^2}{\partial t^2}Y(t) +
 \bigl(R(Y(t),\dc(t))\dc(t)\bigr)^{\perp} + 3 A_tA^*_tY(t)=0.
\]
\end{thm}

One should consider $\bigl(R(\cdot,\dc(t))\dc(t)\bigr)^{\perp} + 3 A_tA^*_t$ 
as the modified curvature operator.  
The crucial point in the equation is that the additional  O'Neill type term
$ 3 A_tA^*_t$  is positive 
semidefinite.

\begin{cor}\label{cor: decomposition}\label{cor: decomp} 
Consider  an $n-1$-dimensional family $\bbJ$ of 
normal Jacobi fields with a self adjoint Riccati operator
 along a geodesic $c\colon \R \rightarrow M$
 in a
nonnegatively curved manifold. 
Then 
\[
\bbJ=\spann_{\R}\bigl\{ J\in \bbJ \mid \mbox{ $J(t)=0$ for some $t$}\bigr\}\oplus 
\bigl\{ J\in \bbJ\mid \mbox{ $J$ is parallel}\bigr\}.
\]
\end{cor}

\subsection{Which bundles occur?}
The major open problem in the subject 
is 

\begin{prob} Let $(\Sigma,g)$ be a nonnegatively curved 
compact manifold. Which vectorbundles  
$E$ over $\Sigma$ admit nonnegatively curved metrics 
such that the zero section of the bundle is a soul?
\end{prob}

If $L$ is a nonnegatively curved compact manifold  
with a free isometric $\Or(k)$ action, then 
the corresponding bundle $L\times_{\Or(k)}\R^k$ 
has a nonnegatively curved metric with the zero section being the 
soul. It is remarkable that all examples of open nonnegatively 
curved manifolds
constructed so far are diffeomorphic to examples arising in this 
way. On the other hand the above method is rather 
flexible already.
From Theorem~\ref{thm: grove ziller} it follows 

\begin{thm}[Grove and Ziller] All vectorbundles over $\Sph^4$ 
admit complete nonnegatively curved metrics.
\end{thm}

It is not known whether one can find nonnegatively 
metrics such that the souls are isometric to the round sphere. 
The souls of the Grove--Ziller metrics have lots zero curvature 
planes. All of the relatively few vectorbundles over $\Sph^5$  
also admit nonnegatively curved metrics [Rigas, 1985]. 
However, in general Cheeger and Gromoll's question 
which bundles over a sphere admit nonnegatively curved metrics remains open.

We mention in some cases one can say a bit more 
about which bundles occur: if either the soul has infinite fundamental group or 
if one fixes the isometry type of the soul. 
\"Ozaydin and Walschap [1994]
 observed that a flat soul necessarily has a flat 
normal bundle. 
If one has an open manifold with infinite fundamental group 
then, by  Theorem~\ref{thm: deforming}
 one can deform the metric within the space of nonnegatively curved
 metrics 
such that a finite cover is isometric to $T^d\times M$,
 where $M$ is simply connected. 
 This in turn shows that the normal bundle 
of the soul $T^d\times \Sigma'$ is canonically isomorphic to the pull back
 of a bundle over the simply connected factor $\Sigma'$. 
The question whether such a bundle can also be written  as 
a twisted bundle over $T^d\times \Sigma'$ was studied in great detail 
by  Belegradek and Kapovitch [2003] using rational homotopy theory.

Moreover one can analyze the situation if the soul is 
isometric to a simply connected  product $\Sigma=\Sigma_1\times \Sigma_2$.
 Although this is just an observation due to the author
 we carry out some details 
here since they can not be found in the literature.
If $u_i\in T_p\Sigma$ is tangent to the $i$-th factor $(i=1,2)$, 
then $R(u_1,u_2)v=0$ for $v\in \nu_p(\sigma)$.
 By ''integrating'' this equation we deduce  
that for a closed curve $\gamma(t)=(\gamma_1(t),\gamma_2(t))$ 
the normal parallel transport $\Par_{\gamma}$ 
decomposes $\Par_{\gamma}=
\Par_{\gamma_1}\circ \Par_{\gamma_2}=\Par_{\gamma_2}\circ \Par_{\gamma_1}$. 
Thus the normal holonomy group is given as the product of two commuting 
subgroups. Each subgroup gives rise to a principle bundle 
over $\Sigma$ which is isomorphic to the pull back bundle 
of a principle bundle over $\Sigma_i$ under the natural projection 
$\Sigma\rightarrow \Sigma_i$.
If we decompose the normal bundle into parallel 
subbundles $\nu(\Sigma)=\nu_1(\Sigma)\oplus \cdots \oplus \nu_l(\Sigma)$ 
such that on each summand the holonomy group 
is irreducible, then each summand is isomorphic 
to a tensor product $\nu_i(\Sigma)=\nu_{i1}(\Sigma)\otimes_{\fK}\nu_{i2}(\Sigma)$
where $\nu_{ij}(\Sigma)$ is isomorphic 
to the pull back of a $\fK$ vectorbundle bundle over $\Sigma_j$ 
under the natural projection $\Sigma \rightarrow \Sigma_j$, $j=1,2$ 
and $\fK\in \{\R,\C,\fH\}$ depends on $i$.

Since any vectorbundle over $\Sph^3$ is trivial, we deduce. 

\begin{cor} Suppose the soul is isometric to a product $\Sph^3\times \Sph^3$ 
then the normal bundle of the soul is trivial.
\end{cor}
\hspace*{1em}

\subsection{ The space of nonnegatively curved metrics.} 
Perelman's theorem indicates that the moduli space 
of metrics should be rather small. On the other hand one can not expect too much.
Belegradek used the method of Grove and Ziller 
to exhibit the following phenomena. 

\begin{thm}[Belegradek] There is a non-compact manifold 
$M$ that admits a sequence of complete nonnegatively curved 
metrics $(g_k)_{k\in \fN}$ such that the souls 
of $(M,g_k)$ are pairwise non-diffeomorphic.
\end{thm}

The theorem shows that 
the moduli space of nonnegatively curved metrics on $M$
has infinitely many components. 
This is in sharp contrast to the space 
of nonnegatively curved metrics on $\Sph^2\times \R^2$.

\begin{thm}[Gromoll and Tapp] Up to a diffeomorphism  a nonnegatively curved 
metric on $\Sph^2\times \R^2$ is either a product metric
or the metric is invariant under the effective action of a two torus 
and it can be obtained as a quotient of a product metric on
 $\Sph^2 \times \R^2\times \R$ by a free $\R$-action.
\end{thm}

For a nontrivial $2$-dimensional vector bundles over $\Sph^2$
the space of nonnegatively curved metrics is somewhat more 
flexible. In fact Walschap [1988] showed that 
given an open four manifold with a soul $\Sph^2$ 
for which any zero curvature plane is tangent to one of the Perelman 
flats from Theorem~\ref{thm: soul con} the following holds:
Let $\tfrac{\partial}{\partial \varphi}$ denote one of the two unit vectorfields
in $M\setminus \Sigma$ tangent to the fibers 
of the Sharafutdinov retraction and whose integral curves 
have constant distance to the soul. If $f$ is an arbitrary function on $M$ 
with compact support contained in $M\setminus \Sigma$, then
the following metric  has nonnegative sectional curvature as well, 
\[
g_t(u,v):=g(u,v)+ tf(p) g\bigl(u,\tfrac{\partial}{\partial \varphi}\bigr)g\bigl(v,\tfrac{\partial}{\partial \varphi}\bigr)
\] 
for all $u,v \in T_pM$ and all
small $t$.

A partial rigidity result was established by Guijarro and Petersen [1997],

\begin{thm}  Let $(M,g)$ be an open nonnegatively curved 
manifold and $p\in M$. Suppose that for any sequence 
$p_n\in M$ converging to $\infty$ the corresponding  sequence $\scal(p_n)$ of scalar curvatures
tends to $0$. Then the soul of $M$ is flat.
\end{thm}

\section{Positively curved manifolds with symmetry}\label{sec: pos}

Grove (1991) suggested to classify manifolds of positive sectional 
curvature with a large isometry group. 
The charm of this proposal is that 
everyone who starts to work on this problem is himself in charge 
of what 'large' means. 
One can relax the assumption if one gets new ideas. 
One potential hope could be that if one understands 
the obstructions for positively manifolds with a 'large' amount of symmetry, one 
may get an idea for a general obstruction.
However the main hope of Grove's program is that the process 
of relaxing the assumptions should lead
toward constructing new examples.

That this can be successful was demonstrated by the 
classification of simply connected homogeneous spaces of positive sectional 
curvature carried out by
Berger [1961], Wallach [1972], Aloff Wallach [1975] and Berard Bergery [1976]. 
The classification led to new examples in dimension $6$, $7$ and $12$, 
$13$ and $24$.
For the sake of completeness it should be said that the only other source of known 
positively curved examples are biquotients, i.e., quotients $\G/\!/\gH$, where 
$\G$ is a compact Lie group and $\gH$ is a subgroup of $\G\times \G$ acting 
freely on $\G$ from the left and the right. 
Eschenburg [1982] and Bazaikin [1996] 
found  infinite series of such examples in dimensions $7$ and $13$. 
We refer the reader to the survey of 
Wolfgang Ziller for more details.

Another motivation for Grove's 
proposal was the following theorem. 

\begin{thm}[Hsiang and Kleiner, 1989]\label{thm: HK}
Let $M^4$ be an orientable compact $4$-manifold of positive sectional curvature. 
Suppose that there is an isometric nontrivial action of $\gS^1$ on $M^4$.
Then $M^4$ is homeomorphic to $\Sph^4$ or $\CP^2$.
\end{thm}
The theorem is a special case of Theorem~\ref{thm: kleiner}.
Grove and Searle [1994] realized that 
the proof of the above theorem can be phrased naturally in terms of
Alexandrov geometry of the orbit space $M^4/\gS^1$. 
A careful analysis of the orbit space 
also allowed them  to establish 
the following result. 
\begin{thm}[Grove and Searle]\label{thm: GS}
Let $M^n$ be  an orientable compact Riemannian manifold of positive sectional curvature. 
Then
\[
\symrank(M,g):=\rank(\Iso(M,g))\le \bigl[\tfrac{n+1}{2}\bigr]
\]
and if equality holds, then $M$ is diffeomorphic to $\Sph^n$,
 $\CP^{n/2}$ or to a lens space.
\end{thm}

The inequality is a simple consequence of a theorem of Berger 
saying that a Killing field on an even dimensional positively curved manifold 
has a zero. 
For the equality discussion Grove and Searle first show, that there in an
isometric $\gS^1$ action on $M$ such that the fixed point 
set has a component $N$ of codimension $2$. 
They then prove that the distance function $d(N,\cdot)$ 
has no critical points in $M\setminus N$ except for precisely one $\gS^1$-orbit 
where it attains its maximum. 
This is used to recover the structure of the manifold. 

Another result which essentially relies on the study 
of the orbit space is due to Rong [2002]. 
He showed that a simply connected positively curved $5$-manifold with symmetry rank 2 
is diffeomorphic to $\Sph^5$.

%
%


Recently, the author made the following basic observation, see [Wilking, 2003].
\begin{thm}[Connectedness Lemma]\label{thm1} \label{thm:connected} Let $M^{n}$ be a compact Riemannian manifold 
with positive sectional curvature. 
\begin{enumerate}
\item[a)] Suppose $N^{n-k}\subset M^{n}$ is a compact totally geodesic 
embedded submanifold. Then the inclusion map $N^{n-k}\rightarrow M^{n}$ 
is $n-2k+1$ connected. If there is a Lie group 
$\G$ that acts isometrically on $M^n$ and fixes $N^{n-k}$ pointwise, then the 
inclusion map is $n-2k+1+\delta(\G)$ connected where 
$\delta(\G)$ is the dimension of the principal orbit.
\item[b)] Suppose $N^{n-k_1}_1,N^{n-k_2}_2\subset M^{n}$ are 
two compact totally geodesic 
embedded submanifolds, $k_1\le k_2$, $k_1+k_2\le n$. Then 
the intersection $N^{n-k_1}_1\cap N^{n-k_2}_2$ is a totally geodesic 
embedded submanifold as well and the inclusion 
\[
N^{n-k_1}_1\cap N^{n-k_2}_2\rightarrow N^{n-k_2}_2
\] 
is $n-k_1-k_2$ connected.
\end{enumerate}
\end{thm}

Theorem~\ref{thm1} turns out to be a very powerful tool in 
the analysis of group actions on positively curved manifolds. 
In fact by combining the theorem with 
the following lemma, one sees that a totally geodesic 
submanifold of low codimension in a positively curved manifold 
has immediate consequences for the cohomology ring of the manifold. 

\begin{lem}\label{lem}\label{lem: mainlem}\label{mainlem} Let $M^n$ be a
closed differentiable oriented manifold, and let 
$N^{n-k}$ be an embedded compact oriented submanifold without boundary.
Suppose the inclusion $ N^{n-k}\rightarrow M^n$ is $n-k-l$ 
connected and $n-k-2l>0$. Let $[N]\in H_{n-k}(M,\Z)$ be the 
image  of the fundamental class of $N$ in $H_*(M,\Z)$ and 
let $e\in H^k(M,\Z)$ be its Poincare dual. 
Then the homomorphism 
\[
\cup e\colon  H^i(M,\Z)\rightarrow H^{i+k}(M,\Z)
 \]
 is surjective for $l\le i< n-k-l$ 
 and  injective for $l< i\le n-k-l$.
\end{lem}

Notice that in the case of a simply connected manifold 
$M$ the submanifold $N$ is simply connected as well and hence 
it is orientable.
Recall that the pull back of $e$ to $H^k(N,\Z)$ 
is the Euler class of the normal bundle of $N$ in $M$.

 Part b) of the Theorem~\ref{thm1} says in particular that 
 $N^{n-k_1}_1\cap N^{n-k_2}_2$ 
is not empty which is exactly the content of Frankel's Theorem. 
In fact similarly to Frankel's Theorem a Synge type 
argument is crucial in the proof of Theorem~\ref{thm1}. 
  The proof of Theorem~\ref{thm1} is a very simple 
Morse theory argument 
in the space of all curves from $N$ to $N$, respectively 
from $N_1$  to $N_2$. The critical points of the energy 
functional are geodesics starting and emanating perpendicularly 
to the submanifolds. Using the second variation formulas
 it is then easy to give
lower bounds on the indices of the nontrivial critical points.

The above result is the main new tool that is used in [Wilking, 2003]
 to show.

\begin{thm}\label{thm2}\label{thm: over4}\label{thm: overfour}
 Let $M^n$ be a simply connected 
$n$-dimensional manifold of positive sectional curvature, $n\ge 8$, 
and let 
$
d\ge 
\tfrac{n}{4}+1.
$ 
Suppose that there is an effective 
isometric action of a torus $\gT^d$ on $M^n$. 
Then $M$ is homotopically equivalent to 
$\CP^{n/2}$ or homeomorphic to $\HP^{n/4}$ or $\Sph^n$.
\end{thm}

In dimensions $8$ and $9$  
the theorem is due to Fang and Rong [2005].
Thus dimensions $6$, $7$ remain the only dimensions where 
one needs maximal symmetry rank assumptions for a classification.

If $M^n$ is an odd-dimensional manifold, that is not simply 
connected but satisfies all other assumptions of the theorem, 
then its fundamental group is cyclic, see Rong [2000].
 A conjecture of  Mann [1965] asserts that an exotic sphere $\Sigma^n$
can not support an 
effective smooth action of a $d$-dimensional torus with $d\ge
\tfrac{n}{4}+1$.

Notice that $\gF_4$, the isometry group of $\CaP$ has rank $4$. 
Thus in dimension $16$ the result is optimal. Similarly 
the isometry group of the $12$-dimensional 
Wallach flag has rank $3$. In dimension $13$ the Berger
space $\SU(5)/\gS^1\cdot \Symp(2)$ is an optimal 
counterexample.
%
%

%

%
%
%
%
There are three major constants to measure the amount of symmetry of
a Riemannian manifold $(M,g)$: 
\[
\symrank(M,g)=\rank\bigl(\Iso(M,g)\bigr),\]\[
\symdeg(M,g)=\dim\bigl(\Iso(M,g)\bigr)\]\[
\cohom(M,g)=\dim\bigl((M,g)/\Iso(M,g)\bigr). 
\]
So far we have mostly considered the first of these constants.

\begin{thm}[Wilking, 2006]\label{mainthm: symdeg} Let $(M^n,g)$ be a simply connected
 Riemannian manifold of 
positive sectional curvature. If $\symdeg(M^n,g)\ge 2n-6$,
then $(M,g)$ is tangentially homotopically equivalent to a rank 1 symmetric space
or isometric to a homogeneous space of positive sectional curvature.
\end{thm}
Notice that
 all homogeneous spaces of positive sectional curvature 
 satisfy the assumptions of the theorem. 
In dimension $7$ the theorem gives the optimal bound 
as there are positively curved Eschenburg space $\SU(3)/\!/\gS^1$ 
with a seven dimensional isometry group. 


Finally we consider the cohomogeneity of a Riemannian manifold.

\begin{thm}[Wilking, 2006]\label{thm}\label{thm: cohom}\label{mainthm: cohom} Let $k$ be a positive integer. 
In dimensions above $18(k+1)^2$ 
each simply connected Riemannian manifold $M^n$ of cohomogeneity 
$k\ge 1$
 with positive sectional curvature is tangentially homotopically equivalent 
to a rank one symmetric space. 
\end{thm}

The proof of Theorem~\ref{thm: cohom} actually establishes  the 
existence of an infinite sequence of (connected) Riemannian manifolds 
\[
M=M_0\subset M_1\subset\cdots 
\] 
such that $\dim(M_i)=n+ih$, where $h\le 4k+4$ is a positive 
integer that depends on $M$. All inclusions are totally geodesic,
all manifolds are   of cohomogeneity 
$k$ and all have positive sectional curvature. 
One then considers $M_\infty:=\bigcup M_i$. 
On the one hand one can use the connectedness lemma 
to show that $M_{\infty}$ has $h$-periodic integral cohomology ring. 
On the other hand,  using Alexandrov 
geometry of the orbit space, one can show that $M_{\infty}$, 
has the homotopy type of the classifying space of a compact Lie group.
The results combined show that $M_{\infty}$ is either contractible
or has the homotopy type of $\CP^\infty$ or $\HP^\infty$. 
The connectedness lemma then implies that $M$ has the corresponding
homotopy type. The details are quite involved and we refer the 
reader to [Wilking, 2006].

Of course one might hope that for small $k$ one can use 
similar techniques to get a classification in all 
dimensions, or at least a classification 
up to some potential candidates for positively curved manifolds.

The following theorem carries out such a program in the case of $k=1$.

\begin{thm}[Verdiani, Grove, Wilking, Ziller]
 Let $M^{n}$ be a simply connected  compact
 Riemannian manifold of positive sectional curvature.
Suppose that a connected Lie group $\G$ acts by isometries with cohomogeneity one ,i.e.,
 the orbit space $M^{n}/\G$ is one dimensional.
Then one of the following holds:
\begin{enumerate} 
\item[$\bullet$] $M^{n}$ is equivariantly diffeomorphic to one 
of the known positively curved biquotients endowed 
with a natural  cohomogeneity action. 
\item[$\bullet$] $n=7$ and $M$ is the two fold cover 
of a $3$-Sasakian manifold that corresponds to one of the 
self dual Einstein $4$-orbifolds of cohomogeneity one 
that were found by Hitchin.
\item[$\bullet$] $n=7$ and $M$ is equivariantly diffeomorphic 
to one particular cohomogeneity one manifold. 
\end{enumerate}
\end{thm}
In even dimensions the theorem is due to Verdiani [2001], 
in this case only rank 1 symmetric spaces occur. 
The odd dimensional case is more involved and is due to 
Grove, Wilking and Ziller [2006]. This is partly due to the fact 
that in dimensions $7$ and $13$, 
there are infinitely many positively curved biquotients 
of cohomogeneity one.

It remains open whether the last two cases can indeed occur. 
The proof of the theorem uses a lot of the techniques that we
have mentioned above. We refer to the survey of Ziller 
for a more detailed discussion.
Very different results on positively curved manifolds with 
symmetry were found by Dessai [2005].

\begin{thm} Suppose $(M,g)$ is a positively curved spin manifold of 
dimension $\ge 12$. Let $\G$ be a connected Lie group acting smoothly 
and suppose a subgroup $\Z_2^2\subset \G$ acts by isometries. 
Then the characteristic number $\hat{A}(M,TM)$ vanishes. 
\end{thm}

The proof is a clever combination of Frankel's theorem 
on the intersection of totally geodesic submanifolds and a vanishing 
theorem of Hirzebruch and Slodowy. 
The non-vanishing of $\hat{A}(M,TM)$  would by that result 
ensure that each of the 
three
involutions in $\Z_2^2$ has a fixed point set of codimension $4$. 
By Frankel these three components have a common intersection
and the contradiction arises by inspecting the isotropy representation 
of $\Z_2^2$ at a fixed point.

In the presence of stronger symmetry assumptions he can show the vanishing 
of more characteristic numbers. These numbers occur naturally 
as coefficients of a power series expanding the elliptic genus.

\subsection{Manifolds with positive sectional curvature almost everywhere.}

As mentioned before there are relatively 
few known examples of positively curved manifolds. 
The lists of examples is quite bit longer 
if one includes manifolds which have 
positive sectional curvature on an open dense set. 

The most interesting example in the class is the 
Gromoll Meyer sphere. Gromoll and Meyer [1974]
considered the subgroup $\gH\subset \Symp(2)\times \Symp(2)$  
given by \[
\gH:=\bigl\{\bigl(\diag(1,q),\diag(q,q)\bigr)\bigm| q\in\gS^3\bigr\}
\]
and the induced free two sided action of $\gH$ on $\Symp(2)$. 
They showed that the corresponding biquotient $\Sigma^7:=\Symp(2)/\!/\gH$ 
is diffeomorphic to an exotic sphere. 

Furthermore, by the O'Neill formulas 
the metric on $\Symp(2)/\!/\gH$ induced by the biinvariant metric $g$ 
on $\Symp(2)$ has nonnegative sectional curvature. 
In fact it is not hard to see that there is a point $p\in \Sigma^7$ 
such that all planes based at $p$ have positive curvature. 

It was shown later by Wilhelm [1996] that there is a left invariant metric 
on $\Symp(2)$ such that the induced metric on $\Sigma^7$ 
has positive sectional curvature on an open dense set of points.
Gromoll and Meyer mention in their paper the so called  
deformation conjecture:

\begin{prob}(Deformation conjecture) Let $M$ be a 
complete nonnegatively curved manifold  for which there is point $p\in M$ such that 
all planes  based at $p$ have positive sectional curvature. 
Does $(M,g)$ admit a positively curved metric, as well? 
\end{prob} 

In the case of an open manifold $M$  the conjecture is by Perelman's 
solution of the soul conjecture valid. 
However in general counterexamples were found in [Wilking, 2002].

\begin{thm}\label{thm: almost pos} The projective tangent bundles 
$P_{\R}T\RP^n$, $P_{\C}T\CP^n$ and $P_{\fH}T\HP^n$ 
of the projective spaces admit metrics with positive sectional curvature
on open dense sets.
\end{thm}

It is easy to see that the projective tangent bundle 
of $P_{\R}T\RP^{2n+1}$ is odd dimensional and not orientable. 
By a theorem of Synge it can not admit a metric with positive 
sectional curvature. 
In dimensions $4n-1$, $(n \ge 3)$ there are infinitely many 
homotopy types of simply connected compact manifolds
with positive sectional curvature on open dense sets. 
In fact one 'half' of the circle bundles  over $P_{\C}T\CP^n$
give rise to such examples.

It is also interesting to note that the natural inclusions 
among these examples remain totally geodesic embeddings and that the isometry 
groups of the manifolds act with cohomogeneity $2$.
By the results on positively manifolds with symmetry,
these properties could not persist for positively curved metrics. 
Another consequence is 
that $\Sph^2\times \Sph^3$, the universal cover of 
$P_{\R}T\RP^3$,
admits a metric with positive sectional curvature on an open dense 
set.

Finally we should mention that prior to [Wilking, 2002],
 Petersen and Wilhelm [1999] constructed a slightly different metric 
on the unit tangent bundle of $\Sph^4$, 
the universal cover of $ P_{\R}T\RP^{4}$, with positive curvature on an open dense set.

\section{ Open Problems.} 
In this final section we mention some of the major 
open problems in the subject. The authors 
favorite conjecture in this context is 
the so called Bott-conjecture which was 
promoted by Grove and Halperin

\begin{conj} Any nonnegatively curved manifold is rationally 
elliptic.
\end{conj}

We recall that a manifold is called rationally elliptic 
if $\pi_*(M)\otimes \Q$ is finite dimensional. 
The conjecture would for example show that
the total rational Betti number of a nonnegatively curved 
manifold $M$ is bounded above by $2^n$
 with equality if and only if $M$ is a flat torus.

There is a conceptual reason why the Bott-conjecture holds 
for all known nonnegatively curved manifolds. Up to deformation of metrics
all known nonnegatively  manifolds are constructed 
from 
Lie groups endowed with biinvariant metrics 
using the following three techniques

\begin{enumerate}
\item[$\bullet$] One can take products of nonnegatively curved manifolds.
\item[$\bullet$] One can pass from a nonnegatively curved 
manifold endowed with a free isometric group action 
to the orbit space endowed with its submersion metric.
\item[$\bullet$] Due to work of Cheeger[1973]
and Grove and Ziller [2000] one can sometimes glue disc bundles, 
i.e., if $M$ is a nonnegatively curved manifold 
which is in two ways the total space of a  
sphere bundle (with the structure group being a Lie group),
 then sometimes the manifold obtained by glueing the two 
corresponding disc bundles has nonnegative curvature as well. 
\end{enumerate}

It is well known that Lie groups are rationally elliptic.
Furthermore, by the exact homotopy sequence the class of rationally 
elliptic manifolds is invariant under taking quotients 
of free actions. 
By the work of Grove and Halperin [1987], 
a manifold obtained by gluing two disc bundles along their 
common boundary is rationally elliptic if and only if
the boundary is.

Grove suggested that the conjecture should hold more 
generally for the class of simply connected almost nonnegatively curved 
manifolds. Here we call a manifold almost nonnegatively curved if it admits 
a sequence $g_k$ of metrics with diameter $1$ and sectional curvature $\ge -\eps(k)\to 0$.
The latter class contains more known examples. 
On the other hand the only additional technique needed 
to construct all of the known simply connected 
almost nonnegatively curved manifolds is: 

\begin{enumerate}
\item[$\bullet$] If $M$ is an almost nonnegatively  curved manifold 
and $P$ is a principal $\G$-bundle over $M$ with $\G$ being a 
compact Lie group, then $P$ has almost nonnegative sectional curvature
as well.
\end{enumerate}

Clearly with this method one can not leave the class of rationally elliptic 
manifolds either. Grove suggested that it might be possible 
to prove the Bott conjecture by induction on the dimension. In this context 
he posed the problem whether any compact nonnegatively curved manifold 
has a nontrivial collapse: Is there a sequence of metrics $g_n$ on 
$M$ with diameter $\le D$ and curvature $\ge -1$ such that $(M,g_n)$ 
converges to a $k$-dimensional Alexandrov space with $0<k<n$. 
Of course it would be also interesting if there is a property
that is shared by all nonnegatively curved Alexandrov spaces, 
and which in the case of manifolds is equivalent to saying that the 
space is rationally elliptic. Alexandrov spaces are more flexible 
since one can take quotients of non free group actions and in the case 
positive curvature joins of spaces. 

Totaro [2003] posed the problem whether any nonnegatively curved manifold 
has a good complexification, i.e., 
is $M$ diffeomorphic to the real points of complex smooth affine variety defined 
over $\R$
such that the inclusion into the complex variety is a homotopy equivalence. 
Totoro's work shows that these manifolds share many properties  
of rationally elliptic manifolds.

The Bott conjecture would also 
imply that the Eulercharacteristic of a nonnegatively curved manifold 
is nonnegative and positive only if the odd rational Betti numbers
vanish. The former statement is part of the Hopf conjecture.

\begin{conj}[Hopf] A compact nonnegatively (positively) 
curved manifold has nonnegative respectively positive 
Eulercharacteristic.
\end{conj}

Slightly more modest (and vague) one might ask 

\begin{question} Is there any obstruction that distinguishes 
the class of simply connected compact manifolds 
admitting nonnegatively curved metrics from the corresponding 
class admitting positively curved metrics?
\end{question}

Of course the huge difference in the number of known examples suggests 
that plenty of such obstructions should exist,
 but to this day there is not a single 
dimension where such an obstruction has been found.
Closely related is 
another Hopf conjecture.

\begin{conj}[Hopf] $\Sph^2\times \Sph^2$ does not admit a metric 
of positive sectional curvature.
\end{conj}
Unlike on $\Sph^2\times \Sph^3$ it is not known 
whether there is metric on $\Sph^2\times \Sph^2$ 
with positive curvature almost everywhere. 
For that reason one could hope that the nonnegatively curved 
metrics on $\Sph^2\times \Sph^2$ are rather rigid. 
In fact a partial confirmation of this view was given by 
 Bourguignon, Deschamps, and Sentenac [1972]. They showed that for a 
product metric on $\Sph^2\times \Sph^2$ without Killing fields
any  analytical deformation which preserves 
nonnegative curvature is up to diffeomorphisms 
 given by a deformation through product metrics.

However, one should be careful to expect too much rigidity in this 
context. The author learned the following observation from 
Bruce Kleiner. We consider $\Sph^2\times \Sph^2$ 
endowed with the M\"uter metric 
\[
(\Sph^2\times \Sph^2,g)= \gS^1\times \gS^1\times\{1\}
\bigl\backslash \SO(3)\times \SO(3)\times \SO(3)\bigr/\Delta \SO(3)
\]
where $\SO(3)^3$ is endowed with a biinvariant metric. 
Clearly the metric is of cohomogeneity one, since there 
is an  $\SO(3)$-action  from the left on $\SO(3)^3$ commuting 
with the left action of $\gS^1\times \gS^1$.
The two singular orbits are given by two $2$-dimensional spheres
 and we let 
$M_{reg}\subset 
\Sph^2\times \Sph^2$ denote the union of all principal orbits.
M\"uter [1987] showed for each point $p\in M_{reg}$
that there is precisely one 
zero curvature plane based at $p$. 
Moreover the plane is tangent to a totally geodesic torus in $M$.

In particular the generic part of the manifold 
$M_{reg}\subset \Sph^2\times \Sph^2$ 
is foliated by totally geodesic flat submanifolds. 
We now consider a symmetric $(2,0)$ tensor $b$,
 whose compact support is contained in $M_{reg}$ and for which 
$b(v,\cdot)=0$ for all $v$ contained 
in a zero curvature plane. It is then straightforward to 
check that the foliation of $M_{reg}$ 
by totally geodesic flats remains a totally geodesic foliation 
by flats for all metrics in the family 
$g(t)=g+tb$. Therefore, the set of zero curvature 
planes of $(M,g(t))$ contains the set of zero curvature planes 
in $(M,g(0))$. What is more: the 
zero curvature planes remain critical points of the 
sectional curvature. Since the zero curvature 
planes in $(M_{reg},g(0))$ form a submanifold of the Grassmannian $\Gr_2(M_{reg})$
and the Hessian of the sectional curvature function 
is nondegenerate transversal to this submanifold, 
it is clear that $(M,g(t))$ has nonnegative 
sectional curvature for all small $t$. 

This shows that the space of nonnegatively curved metrics of
$\Sph^2\times \Sph^2$ is 
somewhat larger than one would expect at first glance.\\[1ex]

One way to give new impulses 
to the subject is to construct new examples. 
In this context we pose the following question. 

\begin{question} Are there any positively curved 
compact Alexandrov spaces satisfying Poincare duality which 
are not homeomorphic to one of the known positively curved manifolds?
\end{question}

Of course an easy way to check that an Alexandrov space satisfies Poincare 
duality  
 is to show that the space of directions at each point is homeomorphic to 
a sphere. 
One could try to look at non free isometric group actions 
on nonnegatively curved manifolds and ask whether the orbit space 
is homeomorphic to a manifold without boundary. 
It would be also  interesting to know, whether in special circumstances 
one can resolve the metric singularities of a positively curved Alexandrov space 
while keeping positive curvature.

\section{Added in Proof.}\label{sec: added} 

One of the most significant developments in the subject
 took place after this survey was completed. 
We will briefly explain it here.
We recall that a manifold is strictly pointwise quarter pinched 
if at each point $p\in M$  there is a constant $\kappa(p)\ge 0$ such that 
for all planes based at $p$ have curvature strictly between $\kappa(p)$ 
and $4\kappa(p)$. 

\begin{thm}[Brendle and Schoen, 2007]\label{thm: quarter ricci} For any strictly pointwise 
quarter pinched manifold $(M,g)$, the normalized  Ricci flow evolves 
$g$ to a limit metric of constant sectional curvature. 
\end{thm}

 We use the notation that 
we introduced in section 1 in connection with  Theorem~1.10.
The theorem relies on the following 
result.

\begin{thm}[B\"ohm and Wilking, 2006]\label{cones} Let $C$ be an $\Or(n)$-invariant cone $C$ 
in the vector space of curvature operators $S^2_B(\so(n))$
 with the following properties 
\begin{enumerate}
\item[$\bullet$] $C$ is invariant under the ODE $\tfrac{d}{dt}R=R^2+R^\#$. 
\item[$\bullet$] $C$ contains the cone of nonnegative curvature operators 
or slightly weaker all nonnegative curvature operators of rank 1.
\item[$\bullet$] $C$ is contained in the cone of curvature operators 
with nonnegative sectional curvature. 
\end{enumerate}
Then for any compact manifold $(M,g)$ whose curvature operator 
is contained in the interior of $C$ at every point $p\in M$,
 the normalized Ricci flow 
evolves $g$ to a limit metric of constant sectional curvature. 
\end{thm} 

It actually suffices to assume that the curvature operator of $(M,g)$ 
is contained in $C$  at all points, and 
 in the interior of $C$ at some point, cf. [Ni and Wu, 2006].

We should remark that the theorem was not stated like 
this in [B\"ohm and Wilking 2006]. However 
by Theorem 5.1 in that paper it suffices to construct a pinching family with $C(0)=C$. Furthermore, the construction of 
a pinching family  for the cone of nonnegative 
curvature operators only relied on the above three properties. 
In other words, one can define a pinching family $C(s)$ with $C(0)=C$ 
by 
\[
C(s):= l_{a(s),b(s)}\bigl(\{ R\in C\mid \Ric\ge \tfrac{\scal}{n}p(s)\}\bigr) 
\] 
where the parameters $a(s),b(s)$ defining 
the linear map $l_{a(s),b(s)}\colon S^2_B(\so(n))\rightarrow S^2_B(\so(n))$ 
and  $p(s)$ are chosen exactly as in [B\"ohm and Wilking, 2006].

\begin{proof}[Sketch of the proof of Theorem~\ref{thm: quarter ricci}.]
The most important step was proved independently 
by Nguyen [2007] and Brendle and Schoen [2007]: Nonnegative isotropic 
curvature defines a Ricci flow invariant curvature condition. 
Both proofs are similar. By Hamilton's maximum principle 
it suffices to show that the cone $C$  of curvature operators 
with nonnegative isotropic curvature is invariant under the ODE
\[
\tfrac{d}{dt}R=R^2+R^\#. 
\]
The idea is to make use of the second variation formula at a
 four frame where the isotropic curvature attains a zero -- 
that is one uses the fact that the Hessian of the isotropic curvature 
function is positive semidefinite. 
Although the computation is elementary it is quite long and that it succeeds 
comes close to being a miracle.

 Brendle and Schoen then proceed as follows. 
They consider the condition that a Riemannian manifold crossed 
with $\R^2$ has nonnegative isotropic curvature. 
It is easy to see that the cone $C$ of 
 curvature operators corresponding to this 
curvature condition satisfies the hypothesis 
of Theorem~\ref{cones}. 

Finally Brendle and Schoen establish that any pointwise 
quarter pinched manifold $(M,g)$ has the property that 
$(M,g)\times \R^2$ has nonnegative isotropic curvature. 
This is again a lengthly computation.
\end{proof}

\begin{rem} Ni and Wolfson [2007] observed that the condition 
that the manifold crossed with $\R^2 $ has nonnegative isotropic curvature 
is equivalent to saying that $(M,g)$ has nonnegative complex sectional 
curvature. 
They also give an alternative shorter  argument 
that nonnegative complex curvature is preserved by the Ricci flow. 
Finally they remark that the statement that quarter pinched manifolds
 have nonnegative complex curvature was essentially known. 
In fact Yau and Zheng  showed
 in a different context that a curvature operator 
with sectional curvature between $-4$ and $-1$ has nonpositive 
complex sectional curvature. 
\end{rem}

\addtocontents{toc}{}

\hspace*{1em}\\
\begin{footnotesize}
\hspace*{0.3em}{\sc Universit\"at  M\"unster,
Einsteinstrasse 62, 48149 M\"unster, Germany }\\
\hspace*{0.3em}{\em E-mail address: }{\sf wilking@math.uni-muenster.de}\\
\end{footnotesize}
\end{document}